\documentclass[12pt]{amsart}
\usepackage{amssymb}
\usepackage[all]{xy}

\newtheorem{theorem}{Theorem}[section]
\newtheorem{definition}[theorem]{Definition}
\newtheorem{proposition}[theorem]{Proposition}
\newtheorem{conjecture}[theorem]{Conjecture}

\begin{document}

\title[Higher orbitals]{Higher orbitals of quizzy quantum group actions}

\author[T. Banica]{Teodor Banica}
\address{T.B.: Department of Mathematics, University of Cergy-Pontoise, F-95000 Cergy-Pontoise, France. {\tt teo.banica@gmail.com}}

\begin{abstract}
The hyperoctahedral group $H_N$ is known to have two natural liberations: the ``good'' one $H_N^+$, which is the quantum symmetry group of $N$ segments, and the ``bad'' one $\bar{O}_N$, which is the quantum symmetry group of the $N$-hypercube. We study here this phenomenon, in the general ``quizzy'' framework, which covers the various liberations and twists of $H_N,O_N$. Our results include: (1) an interpretation of the embedding $\bar{O}_N\subset S_{2^N}^+$, as corresponding to the antisymmetric representation of $O_N$, (2) a study of the liberations of $H_N$, notably with the result $<H_N^+,\bar{O}_N>=O_N^+$, and (3) a comparison of the $k$-orbitals for the inclusions $H_N\subset H_N^+$ and $H_N\subset\bar{O}_N$, for $k\in\mathbb N$ small. 
\end{abstract}


\maketitle

\section*{Introduction}

According to the quantum algebra theory, the twisted analogue of the commutation relations $ab=ba$ between the coordinates $x_1,\ldots,x_N$ of our ambient space $\mathbb R^N$ is:
$$ab=\begin{cases}
-ba&{\rm for}\ a\neq b\\
ba&{\rm otherwise}
\end{cases}$$

At the matrix level, the twisted analogue of the commutation relations $ab=ba$ between the standard coordinates $u_{11},\ldots,u_{NN}$ of the matrix algebra $M_N(\mathbb R)$ is:
$$ab=\begin{cases}
-ba&{\rm for}\ a\neq b\ {\rm on\ the\ same\ row\ or\ column\ of\ }u\\
ba&{\rm otherwise}
\end{cases}$$

These latter relations $\mathcal R$ can be used in order to construct a twisted analogue of the orthogonal group $O_N$, as abstract spectrum of the following universal algebra:
$$C(\bar{O}_N)=C^*\left((u_{ij})_{i,j=1,\ldots,N}\Big|u=\bar{u},u^t=u^{-1},\mathcal R\right)$$

Generally speaking, the structure of $\bar{O}_N$ is quite similar to that of $O_N$, with the correspondence $O_N\leftrightarrow\bar{O}_N$ being best understood via Schur-Weyl twisting, or via a cocycle deformation method. There are many algebraic, geometric, analytic and probabilistic results which can be obtained in this way, and all this material is quite standard.

One interesting feature of $\bar{O}_N$, however, which escapes the philosophy of the correspondence $O_N\leftrightarrow\bar{O}_N$, is that this appears as quantum symmetry group of the standard hypercube in $\mathbb R^N$. This phenomenon was discovered about 10 years ago, in \cite{bbc}, and has been since the subject of various investigations, with variable degree of success.

Our purpose here is to advance on this question, with a number of new results regarding the action $\bar{O}_N\curvearrowright\{1,\ldots,2^N\}$, and its relation with the action $H_N\curvearrowright\{1,\ldots,2^N\}$, notably by using the recent quantum orbital theory from \cite{ba3}, \cite{bi2}, \cite{lmr}.

\bigskip

As a general framework, we use the theory of easy quantum groups \cite{bsp}, \cite{rwe}, in its modified ``quizzy'' version, from \cite{ba1}, \cite{ba2}. The idea is that any intermediate easy quantum group $H_N\subset G\subset O_N^+$ can be $q$-deformed at $q=-1$, into a certain intermediate quantum group $H_N\subset\bar{G}\subset O_N^+$. The easy quantum groups and their twists $H_N\subset\dot{G}\subset O_N^+$ are $q$-easy in some natural sense, with $q=\pm1$, and are called ``quizzy''.

The quizzy quantum groups can be fully classified, by starting with the list in \cite{rwe}, and twisting. According to \cite{ba2}, these quantum groups are as follows:
$$
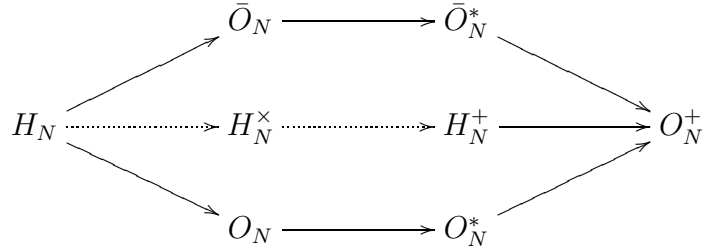
\xymatrix@R=7mm@C=20mm{
&\bar{O}_N\ar[r]&\bar{O}_N^*\ar[rd]\\
H_N\ar@.[r]\ar[ur]\ar[rd]&H_N^\times\ar@.[r]&H_N^+\ar[r]&O_N^+\\
&O_N\ar[r]&O_N^*\ar[ru]}$$

To be more precise, here $O_N^\times$ are the various versions of $O_N$, obtained via liberation and twisting, and $H_N^\times$ are various versions of the hyperoctahedral group $H_N$, which are known to be equal to their own twists. There are in fact many such quantum groups $H_N^\times$, as explained in \cite{rwe}, and the dotted arrows in the middle stand for that.

\bigskip

Now back to our questions, the above diagram, fully covering the liberations and twists of $H_N,O_N$, is precisely what we need. We will study the quantum groups appearing there, from an algebraic and probabilistic point of view, our results being as follows:
\begin{enumerate}
\item We will first study the representation $\bar{O}_N\subset S_{2^N}^+$, our  main result here being that this corresponds to the antisymmetric representation of $O_N$. We will compute as well the magic unitary matrix of this representation, and its character.

\item The fact that $H_N$ has at least two liberations, namely $H_N^+$ and $\bar{O}_N$, raises a number of theoretical questions, regarding the concept of liberation. We will advance here with a negative result, stating that we have $<H_N,\bar{O}_N>=O_N^+$. 

\item The known quizzy quantum group actions on finite sets are those coming from the embeddings $H_N^+\subset S_{2N}^+$ and $\bar{O}_N\subset S_{2^N}^+$. We will compare here the $k$-orbitals for the liberation inclusions $H_N\subset H_N^+$ and $H_N\subset\bar{O}_N$, for $k\in\mathbb N$ small. 
\end{enumerate}

There are of course many questions left. As an interesting algebraic question, for instance, we conjecture that the above embeddings $H_N^+\subset S_{2N}^+$ and $\bar{O}_N\subset S_{2^N}^+$ are the unique ones which can make a quizzy quantum group a quantum permutation one.

\bigskip

We refer to the body of the paper for the precise formulation of the results, which are often quite technical, and sometimes rely on some new notions, to be introduced here.

\bigskip

The paper is organized as follows: 1-2 contain various preliminaries and generalities, in 3-4 we study the quantum permutation representation of $\bar{O}_N$, in 5-6 we discuss liberation questions, in connection with the notion of quantum orbitals, and in 7-8 we present our counting results for orbitals, and we discuss some open questions.

\bigskip

\noindent {\bf Acknowledgements.} I would like to Poulette and her friends, for sharing with me some of their knowledge, and for general advice and support. 

\section{Twisted orthogonality}

We use Woronowicz's quantum group formalism in \cite{wo1}, \cite{wo2}, under the extra assumption $S^2=id$. To be more precise, the definition that we will need is:

\begin{definition}
Assume that $(A,u)$ is a pair consisting of a $C^*$-algebra $A$, and a unitary matrix $u\in M_N(A)$ whose coefficients generate $A$, such that the formulae
$$\Delta(u_{ij})=\sum_ku_{ik}\otimes u_{kj}\quad,\quad \varepsilon(u_{ij})=\delta_{ij}\quad,\quad S(u_{ij})=u_{ji}^*$$
define morphisms of $C^*$-algebras $\Delta:A\to A\otimes A$, $\varepsilon:A\to\mathbb C$, $S:A\to A^{opp}$. We write then $A=C(G)$, and call $G$ a compact matrix quantum group.
\end{definition}

The basic examples are the compact Lie groups, $G\subset U_N$. Indeed, given such a group we can set $A=C(G)$, and let $u_{ij}:G\to\mathbb C$ be the standard coordinates, $u_{ij}(g)=g_{ij}$. The axioms are then satisfied, with $\Delta,\varepsilon,S$ being the functional analytic transposes of the multiplication $m:G\times G\to G$, unit map $u:\{.\}\to G$, and inverse map $i:G\to G$.

There are many other interesting examples of such quantum groups, for the most going back to \cite{wo1}, \cite{wo2}. For a general introduction to the subject, in connection with what we will be doing here, we refer to the book \cite{ntu}, or to the lecture notes \cite{ba4}.

The following key construction is due to Wang \cite{wa1}:

\begin{proposition}
We have a compact quantum group $O_N^+$, defined as follows:
$$C(O_N^+)=C^*\left((u_{ij})_{i,j=1,\ldots,N}\Big|u=\bar{u},u^t=u^{-1}\right)$$
This quantum group contains $O_N$, and the inclusion $O_N\subset O_N^+$ is not an isomorphism.
\end{proposition}

\begin{proof}
It is routine to check that if a matrix $u=(u_{ij})$ is orthogonal ($u=\bar{u},u^t=u^{-1}$), then so are the matrices $u^\Delta=(\sum_ku_{ik}\otimes u_{kj})$, $u^\varepsilon=(\delta_{ij})$, $u^S=(u_{ji})$. Thus we can construct $\Delta,\varepsilon,S$ as in Definition 1.1, by using the universal property of $C(O_N^+)$. 

Regarding the last assertion, we have indeed a quotient map $C(O_N^+)\to C(O_N)$. Now observe that if we denote by $g_1,\ldots,g_N$ the generators of $\mathbb Z_2^{*N}$, the matrix $v=diag(g_i)$ is orthogonal. Thus we have as well a quotient map $C(O_N^+)\to C^*(\mathbb Z_2^{*N})$, the algebra $C(O_N^+)$ follows to be not commutative, and so $O_N\neq O_N^+$, as claimed. See \cite{wa1}.
\end{proof}

We are interested here in the twisted orthogonal group $\bar{O}_N$. In order to introduce this quantum group, we can proceed as follows:

\begin{proposition}
We have the following results:
\begin{enumerate}
\item The usual orthogonal group $O_N\subset O_N^+$ is obtained by assuming that the standard coordinates $u_{ij}$ pairwise commute.

\item We have a quantum group $\bar{O}_N\subset O_N^+$, obtained by assuming that the coordinates $u_{ij}$ anticommute on the rows and columns of $u$, and commute otherwise.
\end{enumerate}
\end{proposition}

\begin{proof}
Once again, this is something standard, the proof being as follows:

(1) Consider the quantum group $O_N'\subset O_N^+$ obtained by dividing $C(O_N^+)$ by its commutator ideal. We have then $O_N\subset O_N'$. On the other hand, $O_N'$ must be a classical space, and by using the coordinates $u_{ij}$ we obtain $O_N'\subset O_N$. Thus $O_N=O_N'$, as claimed.

(2) This follows as in the proof of Proposition 1.2 above, the idea being that if $u=(u_{ij})$ satisfies the relation in the statement, then so do the matrices $u^\Delta=(\sum_ku_{ik}\otimes u_{kj})$, $u^\varepsilon=(\delta_{ij})$, $u^S=(u_{ji})$. Thus we can indeed construct $\Delta,\varepsilon,S$, as claimed. 
\end{proof}

Summarizing, $\bar{O}_N$ appears as a kind of ``twisted counterpart'' of $O_N$. In order to further comment on the definition of $\bar{O}_N$, and on the correspondence $O_N\leftrightarrow\bar{O}_N$, best is to use quantum isometries. Let us recall that the free real sphere $S^{N-1}_{\mathbb R,+}$ is the noncommutative compact space having as coordinates self-adjoint variables $x_1,\ldots,x_N$, subject to the relation $\sum_ix_i^2=1$. We have the following definition, coming from \cite{ba2}, \cite{gos}:

\begin{definition}
A closed subgroup $G\subset O_N^+$ is said to be acting affinely on an algebraic submanifold $X\subset S^{N-1}_{\mathbb R,+}$ when we have a morphism of algebras, as follows:
$$\Phi:C(X)\to C(G)\otimes C(X)\quad,\quad x_i\to\sum_ju_{ij}\otimes x_j$$
The biggest closed quantum group $G\subset O_N^+$ acting affinely on $X$ is called affine orthogonal quantum isometry group of $X$, and is denoted $G^+(X)$.
\end{definition}

It is elementary to check that the usual sphere $S^{N-1}_\mathbb R\subset S^{N-1}_{\mathbb R,+}$ appears by imposing to the coordinates the relations $ab=ba$. So, let us define a twisted sphere $\bar{S}^{N-1}_\mathbb R\subset S^{N-1}_{\mathbb R,+}$ by assuming that the coordinates are subject to the following relations:
$$ab=\begin{cases}
-ba&{\rm for}\ a\neq b\\
ba&{\rm otherwise}
\end{cases}$$

We have the following result, which provides an alternative definition for $\bar{O}_N$:

\begin{proposition}
We have a diagram as follows,
$$
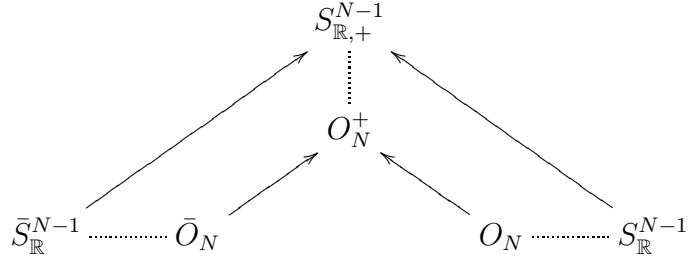
\xymatrix@R=7mm@C=10mm{
&&S^{N-1}_{\mathbb R,+}\\
&&O_N^+\ar@{.}[u]\\
\bar{S}^{N-1}_{\mathbb R}\ar[uurr]\ar@{.}[r]&\bar{O}_N\ar[ur]&&O_N\ar[ul]\ar@{.}[r]&S^{N-1}_\mathbb R\ar[uull]
}$$
with the dotted lines standing for the quantum isometry group construction.
\end{proposition}

\begin{proof}
Here the result for $S^{N-1}_{\mathbb R,+}$ is clear from definitions, the result for $S^{N-1}_\mathbb R$ is known since \cite{bgo}, and follows by using some algebraic tricks, and the result for $\bar{S}^{N-1}_\mathbb R$ can be proved in a similar way, by using the same tricks. For details on all this, see \cite{ba1}.
\end{proof}

Summarizing, we have a quite reasonable understanding of the definition of $\bar{O}_N$. There are as well some more advanced ways of understanding the correspondence $O_N\leftrightarrow\bar{O}_N$, via cocycle twists \cite{bbc}, or via Schur-Weyl twists \cite{ba1}. We will be back to this, later.

Let us review now the quantum isometry result in \cite{bbc}, which is something quite non-standard. Consider the group $\mathbb Z_2^N=<g_1,\ldots,g_N>$. Since the standard generators $g_i\in C^*(\mathbb Z_2^N)=C(\widehat{\mathbb Z_2^N})$ satisfy $g_i=g_i^*,g_i^2=1$, we have an embedding $\widehat{\mathbb Z_2^N}\subset S^{N-1}_{\mathbb R,+}$ given by $z_i=\frac{g_i}{\sqrt{N}}$. Moreover, the image of this embedding is the standard hypercube $K_N$.

With these notions in hand, we can review the result in \cite{bbc}, as follows:

\begin{theorem}
With $\mathbb Z_2^N=<g_1,\ldots,g_N>$ we have a coaction map
$$\Phi:C^*(\mathbb Z_2^N)\to C(\bar{O}_N)\otimes C^*(\mathbb Z_2^N)\quad,\quad g_i\to\sum_ju_{ij}\otimes g_j$$
which makes $\bar{O}_N$ the quantum isometry group of the hypercube $K_N=\widehat{\mathbb Z_2^N}$, as follows:
\begin{enumerate}
\item With $K_N$ viewed as an algebraic manifold, $K_N\subset S^{N-1}_\mathbb R\subset S^{N-1}_{\mathbb R,+}$.

\item With $K_N$ viewed as a graph, with $2^N$ vertices and $2^{N-1}N$ edges.

\item With $K_N$ viewed as a metric space, with the distance from $\mathbb R^N$.
\end{enumerate}
\end{theorem} 

\begin{proof}
This result is basically from \cite{bbc}, the proof being as follows:

(1) In order for $G\subset O_N^+$ to act affinely on $K_N$, the variables $G_i=\sum_ju_{ij}\otimes g_j$ must satisfy the same relations as the generators $g_i\in \mathbb Z_2^N$. The self-adjointness being automatic, the relations to be checked are therefore $G_i^2=1,G_iG_j=G_jG_i$. We have:
\begin{eqnarray*}
G_i^2&=&\sum_{kl}u_{ik}u_{il}\otimes g_kg_l=1+\sum_{k<l}(u_{ik}u_{il}+u_{il}u_{ik})\otimes g_kg_l\\
\left[G_i,G_j\right]&=&\sum_{k<l}(u_{ik}u_{jl}-u_{jk}u_{il}+u_{il}u_{jk}-u_{jl}u_{ik})\otimes g_kg_l
\end{eqnarray*}

By working out these relations, this gives $G\subset\bar{O}_N$, as claimed.

(2) This follows from the fact that $K_N$ is the Cayley graph of $\mathbb Z_2^N$.

(3) Indeed, the metric on $K_N$ comes from the Cayley graph structure.
\end{proof}

The above result was, at the time of \cite{bbc}, and since then, something quite surprizing. One of our purposes in what follows will be that of answering the following question: is the above result something exceptional, or is it part of some general theory?

\section{Representation theory}

In this section we discuss the representation theory of $\bar{O}_N$, and its connection with the representation theory of the hyperoctahedral group $H_N$, and with some other related quantum groups. We will heavily rely on the Tannakian techniques introduced by Woronowicz in \cite{wo2}, further explained in \cite{mal}, \cite{ntu}, and in the lecture notes \cite{ba4}.

The general idea is that since $\bar{O}_N$ appears as a kind of $q=-1$ twist of $O_N$, in a sense close to the one of Drinfeld \cite{dri} and Jimbo \cite{jim}, its representation theory should be very similar to that of $O_N$. As a first illustration for this principle, we have:

\begin{proposition}
The irreducible representations of $\bar{O}_N$ are in one-to-one correspondence with those of $O_N$, with the fusion rules being the same.
\end{proposition}

\begin{proof}
This result is from \cite{bbc}, with the proof using a cocycle twisting interpretation of $\bar{O}_N$, along with a Morita equivalence type argument. See \cite{bbc}. 
\end{proof}

The above result is of course not very explicit, and in addition the fusion rules in question, meaning those for $O_N$, are something quite complicated. In practice, better for our purposes will be to use the ``easiness'' philosophy from \cite{bsp}. We first have:

\begin{proposition}
The Tannakian category of $O_N$ is given by
$$Hom(u^{\otimes k},u^{\otimes l})=span\left(T_\pi\Big|\pi\in P_2(k,l)\right)$$
where $P_2$ is the category of pairings, and where $\pi\to T_\pi$ is given by
$$T_\pi(e_{i_1}\otimes\ldots\otimes e_{i_k})=\sum_{j:\ker(^i_j)\leq\pi}e_{j_1}\otimes \ldots\otimes e_{j_l}$$
with $e_1,\ldots,e_N$ being the standard basis of $\mathbb C^N$.
\end{proposition}

\begin{proof}
This is an old theorem of Brauer \cite{bra}. In what follows we will rely on the proof from \cite{bsp} of this result, with categorical input coming from \cite{mal}. See \cite{ba4}.
\end{proof}

Regarding now $\bar{O}_N$, this quantum group is not exactly easy in the sense of \cite{bsp}, but we have the following result, making the link with \cite{bsp}:

\begin{proposition}
The Tannakian category of $\bar{O}_N$ is given by
$$Hom(u^{\otimes k},u^{\otimes l})=span\left(\bar{T}_\pi\Big|\pi\in P_2(k,l)\right)$$
where $P_2$ is the category of pairings, and where $\pi\to\bar{T}_\pi$ is given by
$$\bar{T}_\pi(e_{i_1}\otimes\ldots\otimes e_{i_k})=\sum_{j:\ker(^i_j)\leq\pi}\varepsilon\left(\ker\begin{pmatrix}i_1&\ldots&i_k\\ j_1&\ldots&j_l\end{pmatrix}\right)e_{j_1}\otimes \ldots\otimes e_{j_l}$$
with $\varepsilon:P_{even}\to\{\pm1\}$ being the standard extension of the signature map $S_\infty\to\{\pm1\}$.
\end{proposition}

\begin{proof}
The point here is that, after reproving the Brauer result as in \cite{bsp}, the extension to $\bar{O}_N$ is straightforward, with the $\pm$ signs in the commutation relations between the coordinates of $\bar{O}_N$ producing the $\pm$ signs from the signature, as stated. See \cite{ba1}.
\end{proof}

At a more advanced level now, it is known as well from \cite{ba1} that a Weingarten type integration formula for $\bar{O}_N$ is available, in the spirit of the one in \cite{csn} for $O_N$, with both the Gram and the Weingarten matrices being invariant under twisting. See \cite{ba1}.

In what follows we will be rather interested in purely algebraic consequences of Proposition 2.3. Following \cite{ba1}, \cite{ba2}, let us introduce the following notion:

\begin{definition}
An intermediate quantum group $H_N\subset\dot{G}\subset O_N^+$ is called ``quizzy'', with deformation parameter $q=\pm1$, when its Tannakian category appears as
$$Hom(u^{\otimes k},u^{\otimes l})=span\left(\dot{T}_\pi\Big|\pi\in D(k,l)\right)$$
for a certain category of partitions $D$, where the correspondence $\pi\to \dot{T}_\pi$ is the usual $\pi\to T_\pi$ correspondence at $q=1$, and is the correspondence $\pi\to\bar{T}_\pi$ at $q=-1$. 
\end{definition}

Here the fact that we must assume $H_N\subset\dot{G}$ comes from the fact that the signature map $\varepsilon$ is defined only on $P_{even}$, and not on the whole $P$. Thus, in order for the twisted maps $\bar{T}_\pi$ to be well-defined, we must have $D\subset P_{even}$, which reads $H_N\subset\dot{G}$. See \cite{ba2}.

In order to discuss now the classification of the quizzy quantum groups, we need to introduce more some examples, coming from \cite{ba1}, \cite{bsp}, as follows:

\begin{definition}
We have intermediate compact quantum groups as follows,
$$O_N\subset O_N^*\subset O_N^+\quad,\quad \bar{O}_N\subset\bar{O}_N^*\subset O_N^+$$
obtained by imposing to the standard coordinates of $O_N^+$ the half-commutation relations $abc=cba$, and the twisted half-commutation relations $abc=\pm cba$.
\end{definition}

To be more precise, the signs in above relations $abc=\pm cba$ are by definition those producing an embedding $\bar{O}_N\subset\bar{O}_N^*$. The precise formula of these signs, which is a bit complicated, can be found in \cite{ba1}. As an explanation here, however, let us mention that $O_N^*,\bar{O}_N^*$ appear respectively as the quantum isometry groups of the half-classical spheres $S^{N-1}_{\mathbb R,*},\bar{S}^{N-1}_{\mathbb R,*}$, whose coordinates are subject to the following relations:
$$x_ix_jx_k=x_kx_jx_i\quad,\quad x_ix_jx_k=(-1)^{\delta_{ij}+\delta_{ik}+\delta_{jk}}x_kx_jx_i$$

For more details on the structure of these quantum groups, on the half-liberation operation, and on these latter noncommutative spheres, see \cite{ba1}, \cite{bi3}, \cite{bdu}.

As examples, we have as well the hyperoctahedral group $H_N$, its free version $H_N^+$, constructed in \cite{bbc}, and the various intermediate liberations $H_N\subset H_N^\times\subset H_N^+$. Following \cite{rwe}, the definition and classification of these latter quantum groups is as follows:

\begin{proposition}
Consider the quantum group $H_N^+\subset O_N^+$ obtained via the relations stating that $p_{ij}=u_{ij}^2$ is magic, in the sense that $p_{ij}$ are projections, summing up to $1$ on each row and column. The easy quantum groups $H_N\subset H_N^\times\subset H_N^+$ are then:
\begin{enumerate}
\item $H_N$ and $H_N^+$ themselves, as well as $H_N^*=H_N^+\cap O_N^*$.

\item A higher half-liberation $H_N^*\subset H_N^{[\infty]}\subset H_N^+$, obtained by imposing the relations $abc=0$, for any $a\neq c$ on the same row or column of $u$.

\item An uncountable family of intermediate quantum groups $H_N\subset H_N^\Gamma\subset H_N^{[\infty]}$, obtained from the quotients $\mathbb Z_2^{*\infty}\to\Gamma$ satisfying a certain uniformity condition.

\item A series of intermediate quantum groups $H_N^{[\infty]}\subset H_N^{\diamond k}\subset H_N^+$, obtained via the relations $[a_1\ldots a_{k-2}b^2a_{k-2}\ldots a_1,c^2]=0$.
\end{enumerate}
\end{proposition}

\begin{proof}
Here the fact that $H_N^+$ is indeed a quantum group, and that both $H_N,H_N^+$ are easy, is from \cite{bbc}. As for the classification result, this is from \cite{rwe}.
\end{proof}

Let us mention as well that, in analogy with the decomposition $H_N=S_N\wr\mathbb Z_2$, we have a decomposition $H_N^+=S_N^+\wr_*\mathbb Z_2$, where $S_N^+$ is Wang's quantum permutation group \cite{wa2}, and where $\wr_*$ is a free wreath product in the sense of Bichon \cite{bi1}. See \cite{bbc}.

We have the following classification result, from \cite{ba2}:

\begin{theorem}
The quizzy quantum groups $H_N\subset\dot{G}\subset O_N^+$ are as follows,
$$\xymatrix@R=7mm@C=20mm{
&\bar{O}_N\ar[r]&\bar{O}_N^*\ar[rd]\\
H_N\ar@.[r]\ar[ur]\ar[rd]&H_N^\times\ar@.[r]&H_N^+\ar[r]&O_N^+\\
&O_N\ar[r]&O_N^*\ar[ru]}$$
with $H_N^\times$ standing for the various liberations of $H_N$, which are all self-dual.
\end{theorem}

\begin{proof}
We already know from Proposition 2.3 above that $\bar{O}_N$ is quizzy, appearing as a Schur-Weyl twist of $O_N$, in the sense that $D$ remains the same, and $q=\pm1$ changes. The same happens for $\bar{O}_N^*$ and $O_N^*$, as explained in \cite{ba2}. As for the converse, this follows from the classification work of Raum and Weber in \cite{rwe}, and from a case-by-case analysis of the twists. To be more precise, what happens is that for the quantum groups listed in \cite{rwe} we have $G=\bar{G}$, except for $G=O_N,O_N^*$. For details here, see \cite{ba2}.
\end{proof}

The above abstract considerations extend to the unitary case, but the situation here is considerably more complicated, with the classification of the easy quantum groups, and therefore of the quizzy quantum groups as well, not known yet. See \cite{gro}, \cite{twe}.

There are many interesting twisting questions too in connection with the various generalizations of the easy quantum group formalism, as those in \cite{fr1}, \cite{fr2}. 

\section{Fourier transforms}

Our purpose here is to compute the magic representation of $\bar{O}_N$, and its character. In order to solve this question, we will need:

\begin{proposition}
The Fourier transform over $\mathbb Z_2^N$ is the map
$$\alpha:C(\mathbb Z_2^N)\to C^*(\mathbb Z_2^N)\quad,\quad\delta_{g_1^{i_1}\ldots g_N^{i_N}}\to\frac{1}{2^N}\sum_{j_1\ldots j_N}(-1)^{<i,j>}g_1^{j_1}\ldots g_N^{j_N}$$
with the usual convention $<i,j>=\sum_ki_kj_k$, and its inverse is the map
$$\beta:C^*(\mathbb Z_2^N)\to C(\mathbb Z_2^N)\quad,\quad g_1^{i_1}\ldots g_N^{i_N}\to\sum_{j_1\ldots j_N}(-1)^{<i,j>}\delta_{g_1^{j_1}\ldots g_N^{j_N}}$$
with all the exponents being binary, $i_1,\ldots,i_N,j_1,\ldots,j_N\in\{0,1\}$.
\end{proposition}

\begin{proof}
Observe first that the group $\mathbb Z_2^N$ can be written as follows:
$$\mathbb Z_2^N=\left\{g_1^{i_1}\ldots g_N^{i_N}\Big|i_1,\ldots,i_N\in\{0,1\}\right\}$$

Thus both $\alpha,\beta$ are well-defined, and it is elementary to check that both are morphisms of algebras. We have as well $\alpha\beta=\beta\alpha=id$, coming from the standard formula:
$$\frac{1}{2^N}\sum_{j_1\ldots j_N}(-1)^{<i,j>}=\prod_{k=1}^N\left(\frac{1}{2}\sum_{j_r}(-1)^{i_rj_r}\right)=\delta_{i0}$$

Thus we have indeed a pair of inverse Fourier morphisms, as claimed.
\end{proof}

As an illustration here, at $N=1$, with the notation $\mathbb Z_2=\{1,g\}$, the map $\alpha$ is given by $\delta_1\to\frac{1}{2}(1+g)$, $\delta_g\to\frac{1}{2}(1-g)$ and its inverse $\beta$ is given by $1\to\delta_1+\delta_g$, $g\to\delta_1-\delta_g$.

By using now these Fourier transforms, we obtain following formula:

\begin{proposition}
The magic unitary for the embedding $\bar{O}_N\subset S_{2^N}^+$ is given by
$$w_{i_1\ldots i_N,k_1\ldots k_N}=\frac{1}{2^N}\sum_{j_1\ldots j_N}\sum_{b_1\ldots b_N}(-1)^{<i+k_b,j>}\left(\frac{1}{N}\right)^{\#(0\in j)}u_{1b_1}^{j_1}\ldots u_{Nb_N}^{j_N}$$
where $k_b=(k_{b_1},\ldots,k_{b_N})$, with respect to multi-indices $i,k\in\{0,1\}^N$ as above.
\end{proposition}

\begin{proof}
By composing the coaction map $\Phi$ from Theorem 1.6 with the above Fourier transform isomorphisms $\alpha,\beta$, we have a diagram as follows:
$$\xymatrix@R=20mm@C=25mm{
C^*(\mathbb Z_2^N)\ar[r]^\Phi&C(\bar{O}_N)\otimes C^*(\mathbb Z_2^N)\ar[d]^{id\otimes\beta}\\
C(\mathbb Z_2^N)\ar[u]^\alpha\ar@.[r]^\Psi&C(\bar{O}_N)\otimes C(\mathbb Z_2^N)}$$

In order to compute the composition on the bottom $\Psi$, we first recall from Theorem 1.6 above that the coaction map $\Phi$ is defined by the formula $\Phi(g_a)=\sum_bu_{ab}\otimes g_b$, for any $a\in\{1,\ldots,N\}$. Now by making products of such quantities, we obtain the following global formula for $\Phi$, valid for any exponents $i_1,\ldots,i_N\in\{1,\ldots,N\}$:
$$\Phi(g_1^{i_1}\ldots g_N^{i_N})=\left(\frac{1}{N}\right)^{\#(0\in i)}\sum_{b_1\ldots b_N}u_{1b_1}^{i_1}\ldots u_{Nb_N}^{i_N}\otimes g_{b_1}^{i_1}\ldots g_{b_N}^{i_N}$$

The term on the right can be put in ``standard form'' as follows:
$$g_{b_1}^{i_1}\ldots g_{b_N}^{i_N}=g_1^{\sum_{b_x=1}i_x}\ldots g_N^{\sum_{b_x}i_x}$$

We therefore obtain the following formula for the coaction map $\Phi$:
$$\Phi(g_1^{i_1}\ldots g_N^{i_N})=\left(\frac{1}{N}\right)^{\#(0\in i)}
\sum_{b_1\ldots b_N}u_{1b_1}^{i_1}\ldots u_{Nb_N}^{i_N}\otimes g_1^{\sum_{b_x=1}i_x}\ldots g_N^{\sum_{b_x=N}i_x}$$

Now by applying the Fourier transforms, we obtain the following formula:
\begin{eqnarray*}
&&\Psi(\delta_{g_1^{i_1}\ldots g_N^{i_N}})\\
&=&(id\otimes\beta)\Phi\left(\frac{1}{2^N}\sum_{j_1\ldots j_N}(-1)^{<i,j>}g_1^{j_1}\ldots g_N^{j_N}\right)\\
&=&\frac{1}{2^N}\sum_{j_1\ldots j_N}\sum_{b_1\ldots b_N}(-1)^{<i,j>}
\left(\frac{1}{N}\right)^{\#(0\in j)}
u_{1b_1}^{j_1}\ldots u_{Nb_N}^{j_N}\otimes\beta\left( g_1^{\sum_{b_x=1}j_x}\ldots g_N^{\sum_{b_x=N}j_x}\right)
\end{eqnarray*}

By using now the formula of $\beta$ from Proposition 3.1, we obtain:
\begin{eqnarray*}
\Psi(\delta_{g_1^{i_1}\ldots g_N^{i_N}})
&=&\frac{1}{2^N}\sum_{j_1\ldots j_N}\sum_{b_1\ldots b_N}\sum_{k_1\ldots k_N}\left(\frac{1}{N}\right)^{\#(0\in j)}\\
&&(-1)^{<i,j>}(-1)^{<(\sum_{b_x=1}j_x,\ldots,\sum_{b_x=N}j_x),(k_1,\ldots,k_N)>}\\
&&u_{1b_1}^{j_1}\ldots u_{Nb_N}^{j_N}\otimes\delta_{g_1^{k_1}\ldots g_N^{k_N}}
\end{eqnarray*}

Now observe that, with the notation $k_b=(k_{b_1},\ldots,k_{b_N})$, we have:
$$<(\sum_{b_x=1}j_x,\ldots,\sum_{b_x=N}j_x),(k_1,\ldots,k_N)>=<j,k_b>$$

Thus, we obtain the following formula for our map $\Psi$:
$$\Psi(\delta_{g_1^{i_1}\ldots g_N^{i_N}})=\frac{1}{2^N}\sum_{j_1\ldots j_N}\sum_{b_1\ldots b_N}\sum_{k_1\ldots k_N}(-1)^{<i+k_b,j>}\left(\frac{1}{N}\right)^{\#(0\in j)}u_{1b_1}^{j_1}\ldots u_{Nb_N}^{j_N}\otimes\delta_{g_1^{k_1}\ldots g_N^{k_N}}$$

But this gives the formula in the statement for the corresponding magic unitary, with respect to the basis $\{\delta_{g_1^{i_1}\ldots g_N^{i_N}}\}$ of the algebra $C(\mathbb Z_2^N)$, and we are done.
\end{proof}

Let us compute now the character of $w$. We first have:

\begin{proposition}
The character of the magic representation of $\bar{O}_N$ is given by
$$\chi=\sum_{j_1\ldots j_N}\sum_{b_1\ldots b_N}\left(\frac{1}{N}\right)^{\#(0\in j)}\delta_{j_1,\sum_{b_x=1}j_x}\ldots\delta_{j_N,\sum_{b_x=N}j_x}u_{1b_1}^{j_1}\ldots u_{Nb_N}^{j_N}$$
with binary indices $j_1,\ldots,j_N\in\{0,1\}$, and plain indices $b_1,\ldots,b_N\in\{1,\ldots,N\}$.
\end{proposition}

\begin{proof}
With the formula in Proposition 3.2, the character is:
\begin{eqnarray*}
\chi
&=&\sum_{i_1\ldots i_N}w_{i_1\ldots i_N,i_1\ldots i_N}\\
&=&\frac{1}{2^N}\sum_{i_1\ldots i_N}\sum_{j_1\ldots j_N}\sum_{b_1\ldots b_N}(-1)^{<i+i_b,j>}\left(\frac{1}{N}\right)^{\#(0\in j)}u_{1b_1}^{j_1}\ldots u_{Nb_N}^{j_N}\\
&=&\sum_{j_1\ldots j_N}\sum_{b_1\ldots b_N}\left(\frac{1}{2^N}\sum_{i_1\ldots i_N}(-1)^{<i+i_b,j>}\right)\left(\frac{1}{N}\right)^{\#(0\in j)}u_{1b_1}^{j_1}\ldots u_{Nb_N}^{j_N}
\end{eqnarray*}

The sum in the middle $S_{ijb}$ is a Fourier sum, computed as follows:
\begin{eqnarray*}
S_{ijb}
&=&\frac{1}{2^N}\sum_{i_1\ldots i_N}(-1)^{<i,j>}(-1)^{i_{b_1}j_1+\ldots+i_{b_N}j_N}\\
&=&\frac{1}{2^N}\sum_{i_1\ldots i_N}(-1)^{<i,j>}(-1)^{i_1\sum_{b_x=1}j_x+\ldots+i_N\sum_{b_x=N}j_x}\\
&=&\frac{1}{2^N}\sum_{i_1\ldots i_N}(-1)^{i_1(j_1+\sum_{b_x=1}j_x)+\ldots+i_N(j_N+\sum_{b_x=N}j_x)}\\
&=&\delta_{j_1,\sum_{b_x=1}j_x}\ldots\ldots\delta_{j_N,\sum_{b_x=N}j_x}
\end{eqnarray*}

We therefore obtain the formula in the statement, and we are done.
\end{proof} 

We can fine-tune the formula found above, as follows:

\begin{proposition}
The magic character of $\bar{O}_N$ is given by $\chi=\sum_{r=0}^N\chi_r$, where
$$\chi_r=\frac{1}{N^{N-r}}\sum_{|A|=r}\sum_{b<A}\prod_{a\in A}u_{ab_a}$$
with the first sum being by definition over sets $A\subset\{1,\ldots,N\}$ satisfying $|A|=r$, the second sum being over functions $b:\{1,\ldots,N\}\to\{1,\ldots,N\}$ satisfying the condition
$$b<A\quad:\quad|b^{-1}(p)\cap A|=\chi_A(p)\ ({\rm mod}\ 2),\forall p$$
and with the product being ordered, and written with the convention $b_a=b(a)$.
\end{proposition}

\begin{proof}
We use the formula in Proposition 3.3. With the notation $r=\#(1\in j)$ we obtain a decomposition $\chi=\sum_{r=0}^N\chi_r$ as in the statement, with:
$$\chi_r=\frac{1}{N^{N-r}}\sum_{\#(1\in j)=r}\sum_{b_1\ldots b_N}\delta_{j_1,\sum_{b_x=1}j_x}\ldots\delta_{j_N,\sum_{b_x=N}j_x}u_{1b_1}^{j_1}\ldots u_{Nb_N}^{j_N}$$

Consider now the set $A\subset\{1,\ldots,N\}$ given by $A=\{a|j_a=1\}$. The binary multi-indices $j\in\{0,1\}^N$ satisfying $\#(1\in j)=r$ being in bijection with such subsets $A$, satisfying $|A|=r$, we can replace the sum over $j$ with a sum over such subsets $A$.

We therefore obtain a formula as follows, where $j$ is the index corresponding to $A$:
$$\chi_r=\frac{1}{N^{N-r}}\sum_{|A|=r}\sum_{b_1\ldots b_N}\delta_{j_1,\sum_{b_x=1}j_x}\ldots\delta_{j_N,\sum_{b_x=N}j_x}\prod_{a\in A}u_{ab_a}$$

We must understand now which multi-indices $b\in\{1,\ldots,N\}^N$ really contribute to the sum, in the sense that all the associated Kronecker symbols in the middle are 1. For this purpose, let us identify $b$ with the corresponding function $b:\{1,\ldots,N\}\to\{1,\ldots,N\}$, via $b(a)=b_a$, as in the statement. Then for any $p\in\{1,\ldots,N\}$ we have:
\begin{eqnarray*}
\delta_{j_p,\sum_{b_x=p}j_x}=1
&\iff&\sum_{b_x=p}j_x=j_p\ ({\rm mod}\ 2)\\
&\iff&\sum_{x\in b^{-1}(p)}j_x=j_p\ ({\rm mod}\ 2)\\
&\iff&|b^{-1}(p)\cap A|=\chi_A(p)\ ({\rm mod}\ 2)
\end{eqnarray*}

We conclude that the multi-indices $b\in\{1,\ldots,N\}^N$ which effectively contribute to the sum are those coming from the functions $b:\{1,\ldots,N\}\to\{1,\ldots,N\}$ satisfying the condition $b<A$ from the statement. Thus, we obtain the formula in the statement.
\end{proof}

The above formula for the character is still not our final one. We can indeed further study the condition $b<A$ appearing there, and we are led to:

\begin{theorem}
The magic character of $\bar{O}_N$ is given by $\chi=\sum_{r=0}^N\chi_r$, where
$$\chi_r=\sum_{|A|=r}\sum_{\sigma\in S_N^A}\prod_{a\in A}u_{a\sigma(a)}$$
with the product being ordered, and where $S_N^A=\{\sigma\in S_N|\sigma_{|A^c}=id\}$.
\end{theorem}

\begin{proof}
We use the formula in Proposition 3.4. By splitting the character $\chi=\sum_{r=0}^N\chi_r$ as indicated there, and then by further splitting each $\chi_r$ over the sets $A\subset\{1,\ldots,N\}$ satisfying $|A|=r$, we must prove that for each of these sets we have:
$$\frac{1}{N^{N-r}}\sum_{b<A}\prod_{a\in A}u_{ab_a}=\sum_{\sigma\in S_N^A}\prod_{a\in A}u_{a\sigma(a)}$$

In order to do so, we must construct a certain correspondence $b\to\sigma$, which leaves invariant the product term, and which produces the multiplicity $N^{N-r}$.

We know that the condition $b<A$ corresponds to the following condition:
$$|b^{-1}(p)\cap A|=\chi_A(p)\ ({\rm mod}\ 2),\forall p$$

Now observe that the validity of this condition, and the value of the corresponding product $\prod_{a\in A}u_{ab_a}$ as well, only concerns the restriction $b_{|A}$. Thus, up to a multiplicity of $N^{|A^c|}=N^{N-r}$, we can replace if we want the restriction $b_{|A^c}$ by the identity.

Summarizing, we must prove that we have the following formula:
$$\sum_{b<A,b_{|A^c}=id}\prod_{a\in A}u_{ab_a}=\sum_{\sigma\in S_N^A}\prod_{a\in A}u_{a\sigma(a)}$$

Our claim is that this formula holds indeed, with the correspondence being given by $b=\sigma$. In order to prove this latter fact, what we have to show is that we have:
$$b_{|A^c}=id,|b^{-1}(p)\cap A|=\chi_A(p)\ ({\rm mod}\ 2),\forall p\iff b\in S_N^A$$

Since everything here depends on $A$ only, we can assume if we want that we have $A=\{1,\ldots,N\}$, and the statement to be proved becomes:
$$|b^{-1}(p)|=1\ ({\rm mod}\ 2),\forall p\iff b\in S_N$$

But this is clear, because the implication $\implies$ follows from $|b^{-1}(p)|\geq1$ for any $p$, and the implication $\Longleftarrow$ is trivial. Thus, we have proved our claim, and we are done.
\end{proof}

\section{Magic actions}

We will be interested in what follows in further understanding the magic action of $\bar{O}_N$, and notably in computing the probabilistic distribution of its character, with respect to the Haar measure of $\bar{O}_N$. For this purpose, simplest is to make the link with $O_N$.

In order to do so, we must further study the quantities $\chi_r$ introduced in Theorem 3.5 above. As a first result here, at small or big values of $r$, we have:

\begin{proposition}
The quantities $\chi_r$ from Theorem 3.5 are as follows:
\begin{enumerate}
\item $\chi_0=1$.

\item $\chi_1=\sum_au_{aa}$.

\item $\chi_2=\sum_{a<c}u_{aa}u_{cc}+u_{ac}u_{ca}$.

\item $\chi_{N-1}=\sum_{a=1}^N\sum_{\sigma\in S_N,\sigma(a)=a}u_{1\sigma(1)}\ldots u_{a-1\sigma(a-1)}u_{a+1\sigma(a+1)}\ldots u_{N\sigma(N)}$.

\item $\chi_N=\sum_{\sigma\in S_N}u_{1\sigma(1)}\ldots u_{N\sigma(N)}$.
\end{enumerate}
Also, at $N=2$ we obtain $\chi=1+u_{11}+u_{22}+u_{11}u_{22}+u_{12}u_{21}$.
\end{proposition}

\begin{proof}
We use the formula found in Theorem 3.5 above, namely:
$$\chi_r=\sum_{|A|=r}\sum_{\sigma\in S_N^A}\prod_{a\in A}u_{a\sigma(a)}$$

(1) Here we must have $A=\emptyset$, and the result is clear.

(2) Here we can write $A=\{a\}$, the only permutation $\sigma\in S_N^A$ is the identity, and we obtain the formula in the statement.

(3) Here we can write $A=\{a,c\}$ with $a<c$, there are two permutations $\sigma\in S_N^A$, namely the identity and the transposition $a\leftrightarrow c$, and we obtain the above formula.

(4) Here we can write $A=\{1,\ldots,N\}-\{a\}$, and we obtain the above formula.

(5) Here we must have $A=\{1,\ldots,N\}$, and the result is clear.

At $N=2$ now, the various formulae that we have give $\chi_0=1$, $\chi_1=u_{11}+u_{22}$, $\chi_2=u_{11}u_{22}+u_{12}u_{21}$, and so $\chi=1+u_{11}+u_{22}+u_{11}u_{22}+u_{12}u_{21}$, as claimed.
\end{proof}

Observe that at $N=2$ the variable $\chi_2=\chi-\chi_0-\chi_1$, and so all the variables $\chi_r$ in this case, is a virtual character in the sense of \cite{wo1}. This can be checked as well directly, by applying the comultiplication to the formula $\chi_2=u_{11}u_{22}+u_{12}u_{21}$.

In fact, according to Theorem 1.6, the action of $\bar{O}_N$ must leave invariant  the $N+1$ eigenspaces of the Laplacian of the cube, and one can prove that the above variables $\chi_0,\ldots,\chi_N$ are the precisely characters of the corresponding representations of $\bar{O}_N$.

In what follows, we will be rather interested in identifying these representations with some similar representations of $O_N$. The correspondence will come from:

\begin{proposition}
Consider the $r$-th antisymmetric representation of $O_N$, on the space
$$X_r=span\left(\xi_{i_1\ldots i_r}\Big|i_1<\ldots<i_r\right)\quad,\quad\xi_{i_1\ldots i_r}=\sum_{\sigma\in S_r}\varepsilon(\sigma)e_{i_{\sigma(1)}}\otimes\ldots\otimes e_{i_{\sigma(r)}}$$
of antisymmetric vectors in $(\mathbb C^N)^{\otimes r}$. The character of this representation is given by
$$\chi_r=\sum_{|A|=r}\sum_{\sigma\in S_N^A}\varepsilon(\sigma)\prod_{a\in A}u_{a\sigma(a)}$$
where $S_N^A=\{\sigma\in S_N|\sigma_{|A^c}=id\}$, and where $\varepsilon:S_N\to\{\pm1\}$ is the signature map.
\end{proposition}

\begin{proof}
The fact that $X_r$ is indeed invariant is well-known, and so we have a representation, as stated. In order to compute now the character, observe that for $g\in O_N$ we have:
\begin{eqnarray*}
g^{\otimes r}\xi_{i_1\ldots i_r}
&=&\sum_{\sigma\in S_r}\varepsilon(\sigma)ge_{i_{\sigma(1)}}\otimes\ldots\otimes ge_{i_{\sigma(r)}}\\
&=&\sum_{\sigma\in S_r}\sum_{j_1\ldots j_r}\varepsilon(\sigma)g_{j_1i_{\sigma(1)}}\ldots g_{j_ri_{\sigma(r)}}e_{j_1}\otimes\ldots\otimes e_{j_r}\\
&=&\sum_{j_1\ldots j_r}\left(\sum_{\sigma\in S_r}\varepsilon(\sigma)g_{j_1i_{\sigma(1)}}\ldots g_{j_ri_{\sigma(r)}}\right)e_{j_1}\otimes\ldots\otimes e_{j_r}
\end{eqnarray*}

By using the properties of the signature map, we see that when the indices $j_1,\ldots,j_r$ are not distinct, the corresponding contribution is 0. Thus, we can restrict the sum over distinct indices $j_1,\ldots,j_r$. Moreover, by arranging these indices increasingly, into a sequence $k_1<\ldots<k_r$, we conclude that we must have, for a certain $\tau\in S_r$:
$$j_1=k_{\tau(1)}\quad,\quad\ldots\quad,\quad j_r=k_{\tau(r)}$$

This correspondence between distinct indices $j_1,\ldots,j_r$ and pairs of increasing sequences $k_1<\ldots<k_r$ plus permutations $\tau\in S_r$ being bijective, we conclude that we have:
$$g^{\otimes r}\xi_{i_1\ldots i_r}=\sum_{k_1<\ldots<k_r}\sum_{\tau\in S_r}\left(\sum_{\sigma\in S_r}\varepsilon(\sigma)g_{k_{\tau(1)}i_{\sigma(1)}}\ldots g_{k_{\tau(r)}i_{\sigma(r)}}\right)e_{k_{\tau(1)}}\otimes\ldots\otimes e_{k_{\tau(r)}}$$

Now by taking the scalar product with $\xi_{k_1\ldots k_r}$, we obtain from this:
$$<g^{\otimes r}\xi_{i_1\ldots i_r},\xi_{k_1\ldots k_r}>=\sum_{\sigma,\tau\in S_r}\varepsilon(\sigma\tau)g_{k_{\tau(1)}i_{\sigma(1)}}\ldots g_{k_{\tau(r)}i_{\sigma(r)}}$$

We can now compute the character. With respect to the basis $\{\xi_{i_1\ldots i_r}\}$, we obtain:
$$\chi(g)=\sum_{i_1<\ldots<i_r}\sum_{\sigma,\tau\in S_r}\varepsilon(\sigma\tau)g_{i_{\tau(1)}i_{\sigma(1)}}\ldots g_{i_{\tau(r)}i_{\sigma(r)}}$$

By permuting the terms on the right, and in terms of the permutation $\rho=\sigma\tau^{-1}$, which has the same signature as the permutation $\sigma\tau$ appearing above, we obtain:
$$\chi(g)=\sum_{i_1<\ldots<i_r}\sum_{\rho\in S_r}\varepsilon(\rho)g_{i_1i_{\rho(1)}}\ldots g_{i_ri_{\rho(r)}}$$

Now if we set $A=\{i_1,\ldots,i_r\}$, and we replace $\rho$ by its extension $\sigma\in S_N^A$, obtained by fixing all the points of $A^c$, this gives the formula in the statement.
\end{proof}

Now by comparing Theorem 3.5 and Proposition 4.2, we obtain:

\begin{theorem}
The magic representation of $\bar{O}_N$ corresponds to the antisymmetric representation of $O_N$, via the correspondence coming from Proposition 2.1.
\end{theorem}

\begin{proof}
This follows by comparing the formulae in Theorem 3.5 and Proposition 4.2. Indeed, the twisting operation $O_N\to\bar{O}_N$ makes correspond the following products:
$$\varepsilon(\sigma)\prod_{a\in A}u_{a\sigma(a)}\to\prod_{a\in A}u_{a\sigma(a)}$$

Now by summing over sets $A$ and permutations $\sigma$, we conclude that the twisting operation $O_N\to\bar{O}_N$ makes correspond the following quantities:
$$\sum_{|A|=r}\sum_{\sigma\in S_N^A}\varepsilon(\sigma)\prod_{a\in A}u_{a\sigma(a)}\to \sum_{|A|=r}\sum_{\sigma\in S_N^A}\prod_{a\in A}u_{a\sigma(a)}$$

Thus the character $\chi_r$ computed for $O_N$ corresponds to the character $\chi_r$ computed for $\bar{O}_N$, and by making a sum over $r\in\{0,1,\ldots,N\}$, this gives the result.
\end{proof}

Summarizing, we have now a good understanding of the magic representation of $\bar{O}_N$, that we will use later on. This representation, however, remains quite exceptional, and in relation with all this, we have the following conjecture:

\begin{conjecture}
There are two types of possible actions of the quizzy quantum groups $H_N\subset G\subset O_N^+$ on finite spaces, as follows:
\begin{enumerate}
\item Those coming from $H_N\subset\bar{O}_N\subset S_{2^N}^+$.

\item Those coming from $H_N\subset G\subset H_N^+\subset S_{2N}^+$.
\end{enumerate}
\end{conjecture}

Proving this looks like a heavy algebraic task, because there are many things to be done, which all look non-trivial. To be more precise, taking into account the classification result in Theorem 2.5 above, the precise list of results to be proved is as follows:
\begin{enumerate}
\item First, we must prove that the above magic corepresentation of $\bar{O}_N$ is the unique one. In view of the correspondence $O_N\leftrightarrow\bar{O}_N$, we must first solve a certain categorical problem for $O_N$, involving Young tableaux, and then look for the ``magic'' implementation of the solutions. This is certainly quite non-trivial.

\item Then, we must prove that $\bar{O}_N^*$ has no magic corepresentation at all. Here we can use the isomorphism $P\bar{O}_N^*=P\bar{U}_N$, and so we are led as well to Young tableaux combinatorics, this time coming from $\bar{U}_N$. For the remaining representations, which are not projective, we can use the classification resuls in \cite{bdu}.

\item Finally, we must deal with the actions of the quantum groups $H_N\subset G\subset H_N^+$, on one hand by proving that in the non-classical case, $G\neq H_N$, the only solution comes as $G\subset H_N^+\subset S_{2N}^+$, and on the other hand by proving that in the classical case, $G=H_N$, the only extra solution comes as $H_N\subset S_{2^N}$.
\end{enumerate}

All this is extremely heavy. We believe however that Conjecture 4.4 is a good problem, hiding many interesting things, and definitely worth investigating.

\section{Liberation theory}

We know that the hyperoctahedral group $H_N$ has at least two natural ``liberations'', namely $H_N\subset H_N^+$ and $H_N\subset\bar{O}_N$. Our purpose here is to systematically investigate this phenomenon. We will see that this will naturally lead us into certain questions regarding the higher orbitals of $H_N,H_N^+,\bar{O}_N$, which will require using Theorem 4.3.

Let us begin with the following very general definition:

\begin{definition}
A liberation of a closed subgroup $G\subset O_N$ is an intermediate quantum group $G\subset H\subset O_N^+$ satisfying $H_{class}=G$. Such a liberation is called:
\begin{enumerate}
\item Maximal, when there is no bigger liberation $H\subset H'$.

\item Universal, when it contains any other liberation $H'$.
\end{enumerate}
\end{definition}

As a basic example, let us take $G=O_N$ itself. The liberations of $G$ are then the intermediate quantum groups $O_N\subset O_N^\times\subset O_N^+$, and so $O_N^+$ is universal.

In general, however, the above notions are quite subtle, even for the trivial group $G=\{1\}$, and this because the condition $H_{class}=\{1\}$ is quite poorly understood.

In order to further comment on these questions, let us recall that for an inclusion of orthogonal quantum groups $G\subset H$ the linear spaces $Fix(u^{\otimes l})$ must decrease, when passing from $G$ to $H$, and that $G\subset H$ is proper precisely when one of these spaces decreases strictly. This follows indeed from the Peter-Weyl theory from \cite{wo1}.

In view of this fact, let us introduce as well:

\begin{definition}
Let $G\subset O_N$ be a closed subgroup.
\begin{enumerate}
\item The level of a liberation $G\subset H\subset O_N^+$ is the smallest $l\in\mathbb N$ such that the space $Fix(u^{\otimes l})$ decreases, when passing from $G$ to $H$.

\item The stability level of $G$ is the biggest $k\in\mathbb N$ such that the space $Fix(u^{\otimes k})$ remains fixed, for any liberation $G\subset H\subset O_N^+$.
\end{enumerate}
\end{definition}

Once again, as a basic example, let us take $G=O_N$ itself. The liberations of $G$ being the intermediate quantum groups $O_N\subset O_N^\times\subset O_N^+$, the stability level is $k=3$, coming from the fact that we have $NC_2(k)=P_2(k)$ at $k\leq3$, but not at $k=4$.

In fact, there are conjecturally only two proper liberations of $O_N$, namely the free version $O_N^+$, having level 4, and the half-liberated version $O_N^*$, having level 6.

These observations have a natural generalization to the easy quantum group setting, from \cite{bsp}. The liberation theory for the easy groups was developed there, by using some inspiration from the Weingarten formula \cite{csn}, \cite{wei} and from the Bercovici-Pata bijection from free probability theory \cite{bpa}, \cite{vdn}, the idea being that the passage $G\to G^+$ simply appears by ``removing the crossings'' from the Tannakian category of $G$. See \cite{bsp}. 

We will be interested in what follows only in the ``true'' liberations $G\to G^+$, which are those having the property that the laws of the main characters are related by the Bercovici-Pata bijection. As explained in \cite{bsp}, there are only 4 such liberations, namely those of the groups $O_N,B_N,S_N,H_N$. We refer as well to \cite{ba4} for this material.

With these preliminaries in hand, we have the following result:

\begin{proposition}
Consider the truly liberable orthogonal easy groups, namely the quantum groups $O_N,B_N,S_N,H_N$, and their easy liberations $O_N^+,B_N^+,S_N^+,H_N^+$.
\begin{enumerate}
\item These liberations are universal, in the easy setting.

\item The stability level is $k=3$, once again in the easy setting.
\end{enumerate}
\end{proposition}

\begin{proof}
Here the first assertion follows from \cite{bsp}, or from the classification results from \cite{rwe}. Regarding now the stability level, since in each case we have a universal liberation, this is given by $k=l-1$, where $l$ is the level of the universal liberation.

The point now is that, in each of the cases under consideration, we have $l=4$. Indeed, as explained above, for $G=O_N$ this follows from $NC_2(l)=P_2(l)$ at $l\leq3$, but not at $l=4$. As for $G=B_N,S_N,H_N$, the situation here is similar, because if we denote by $D$ the corresponding category of partitions, which is respectively $D=P_{12},P,P_{even}$, we have $D(l)\subset NC(l)$ at $l\leq3$, but not at $l=4$, because the basic crossing belongs to $D$.
\end{proof}

In the non-easy setting now, the results for $G=O_N$ still hold. However, in what concerns $G=B_N,S_N,H_N$, the problems here become considerably more difficult. Regarding $G=B_N,S_N$, we believe that the easy liberations $G^+=B_N^+,S_N^+$ are universal, but we have no idea on how to approach this problem. The maximality problem, which is in principle a bit simpler, looks equally difficult. In Tannakian terms, we must prove:
$$C\subset span(NC_{12})\ ,\ {\ }_|\in<C,\slash\hskip-2.1mm\backslash>\implies{\ }_|\in C$$
$$C\subset span(NC)\ ,\ {\ }_|\,,\sqcap\hskip-0.7mm\sqcap\in<C,\slash\hskip-2.1mm\backslash>\implies{\ }_|\,,\sqcap\hskip-0.7mm\sqcap\in C$$

These questions are substantially more complicated than those usually solved in the context of the easy quantum groups, as in \cite{bsp}, \cite{rwe}, \cite{twe}, and we have no results.

Let us discuss now the case $G=H_N$, which is the one that we are interested in. As a starting point, we have the following fact, coming from \cite{bbc}:

\begin{proposition}
The hyperoctahedral group $H_N$ has at least two natural liberations, namely $H_N\subset H_N^+$ and $H_N\subset\bar{O}_N$, and neither of them is universal. 
\end{proposition}

\begin{proof}
The fact that we have indeed liberations is known from \cite{bbc}, and follows for instance from the following formula, valid for any finite graph $X$:
$$G^+(X)_{class}=G^+(X)$$

Indeed, with $X$ being the graph formed by $N$ segments we obtain $(H_N^+)_{class}=H_N$, and with $X$ being the $N$-hypercube, we obtain $(\bar{O}_N)_{class}=H_N$.

Regarding now the last assertion, this follows from the fact that we don't have inclusions $H_N^+\subset\bar{O}_N$ or $\bar{O}_N\subset H_N^+$, because the coordinates of either quantum group don't satisfy the relations for the other. This is indeed clear in view of the definitions of these quantum groups. We will obtain this result as well below, as part of something more general.
\end{proof}

In view of the above result, several natural questions appear, as follows:

(1) Are the above liberations maximal? Here we are led into difficult Tannakian questions, of the same flavor as the above-mentioned ones for $B_N,S_N$, namely:
$$C\subset span(NC_{even})\ ,\ \sqcap\hskip-1.6mm\sqcap\hskip-1.6mm\sqcap\in<C,\slash\hskip-2.1mm\backslash>\implies\sqcap\hskip-1.6mm\sqcap\hskip-1.6mm\sqcap\in C$$
$$C\subset span(\bar{P}_2)\ ,\ \sqcap\hskip-1.6mm\sqcap\hskip-1.6mm\sqcap\in<C,\slash\hskip-2.1mm\backslash>\implies\sqcap\hskip-1.6mm\sqcap\hskip-1.6mm\sqcap\in C$$

(2) Are there any other maximal liberations of $H_N$? Once again, this looks liks a quite difficult problem, which is in need of some new ideas.

(3) What is the compact quantum group $<H_N^+,\bar{O}_N>$ generated by the maximal liberations of $H_N$ that we have, namely $H_N^+$ and $\bar{O}_N$, inside $O_N^+$?

(4) What is the level of the liberation $H_N\subset\bar{O}_N$? And, what can be said about the stability level of $H_N$, in the sense of Definition 5.2 above?

In what follows we will solve (3), and then, later on, comment on (4). Regarding (3), our answer is $<H_N^+,\bar{O}_N>=O_N^+$, coming as part of the following general result:

\begin{theorem}
The diagram of the basic quizzy quantum groups, namely
$$
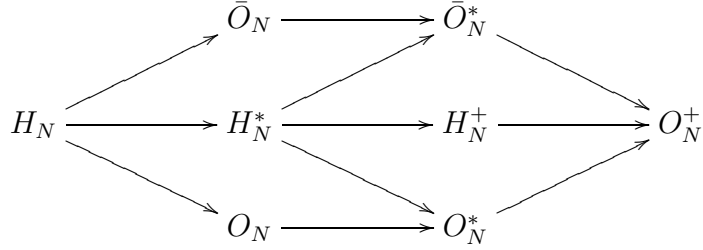
\xymatrix@R=7mm@C=20mm{
&\bar{O}_N\ar[r]&\bar{O}_N^*\ar[rd]\\
H_N\ar[r]\ar[ur]\ar[rd]&H_N^*\ar[ur]\ar[dr]\ar[r]&H_N^+\ar[r]&O_N^+\\
&O_N\ar[r]&O_N^*\ar[ru]}$$
is both an intersection and generation diagram, in the sense that for any square subdiagram $A\subset B,C\subset D$ we have $A=B\cap C$ and $<B,C>=D$.
\end{theorem}

\begin{proof}
The various intersections and generation results are already known, and explained in \cite{ba4}, except for the following two results, that remain to be proved now:
$$<\bar{O}_N,H_N^+>=O_N^+\quad,\quad<\bar{O}_N,H_N^*>=O_N^*$$

In order to prove these two formulae, we use the Tannakian approach. To be more precise, we must prove that we have the following results:
$$span\left(\bar{T}_\pi\Big|\pi\in P_2\right)\bigcap span\left(T_\pi\Big|\pi\in NC_{even}\right)=span\left(T_\pi\Big|\pi\in NC_2\right)$$
$$span\left(\bar{T}_\pi\Big|\pi\in P_2\right)\bigcap span\left(T_\pi\Big|\pi\in P_{even}^*\right)= span\left(T_\pi\Big|\pi\in P_2^*\right)$$

We will only prove the first formula, the proof of the second one being similar. Let us first recall that the M\"obius function of any lattice is given by:
$$\mu(\sigma,\pi)=\begin{cases}
1&{\rm if}\ \sigma=\pi\\
-\sum_{\sigma\leq\tau<\pi}\mu(\sigma,\tau)&{\rm if}\ \sigma<\pi\\
0&{\rm if}\ \sigma\not\leq\pi
\end{cases}$$

With this convention, we have the following formula from \cite{ba2}, which expresses the twisted maps $\bar{T}_\pi$ in terms of the untwised ones $T_\pi$:
$$\bar{T}_\pi=\sum_{\sigma\leq\tau\leq\pi}\varepsilon(\tau)\mu(\sigma,\tau)T_\sigma$$

To be more precise, this formula is valid for any $\pi\in P_{even}$, with the sum being over all partitions $\sigma,\tau\in P_{even}$ satisfying $\sigma\leq\tau\leq\pi$, and with $\mu$ being the M\"obius function of $P_{even}$. We refer to \cite{ba2} for the proof, which follows from the definition of $\dot{T}_\pi$, and from the M\"obius inversion formula. As an illustration, we have the following computation:
\begin{eqnarray*}
\bar{T}_{\slash\hskip-1.5mm\backslash}
&=&\varepsilon(\slash\hskip-2mm\backslash)\mu(\slash\hskip-2mm\backslash,\slash\hskip-2mm\backslash)T_{\slash\hskip-1.5mm\backslash}
+\varepsilon(\slash\hskip-2mm\backslash)\mu(|\hskip-1.9mm-\hskip-1.9mm|,\slash\hskip-2mm\backslash)T_{|\hskip-0.7mm-\hskip-0.7mm|}
+\varepsilon(|\hskip-1.9mm-\hskip-1.9mm|)\mu(|\hskip-1.9mm-\hskip-1.9mm|,|\hskip-1.9mm-\hskip-1.9mm|)T_{|\hskip-0.7mm-\hskip-0.7mm|}\\
&=&(-1)\cdot1\cdot T_{\slash\hskip-1.5mm\backslash}+(-1)\cdot(-1)\cdot T_{|\hskip-0.7mm-\hskip-0.7mm|}+1\cdot1\cdot T_{|\hskip-0.7mm-\hskip-0.7mm|}\\
&=&-T_{\slash\hskip-1.5mm\backslash}+2\,T_{|\hskip-0.7mm-\hskip-0.7mm|}
\end{eqnarray*}

Observe that this agrees with $\bar{T}_{\slash\hskip-1.5mm\backslash}(e_a\otimes e_b)=-e_b\otimes e_a+2\delta_{ab}e_a\otimes e_a$. See \cite{ba2}.

With this formula in hand, let us go back to our problem. By Frobenius duality we can restrict the attention to the fixed vectors, and we want to prove that we have:
$$span\left(\bar{T}_\pi\Big|\pi\in P_2(k)\right)\bigcap span\left(T_\pi\Big|\pi\in NC_{even}(k)\right)=span\left(T_\pi\Big|\pi\in NC_2(k)\right)$$

So, let us pick a vector $\xi$ in the span on the left, as follows:
$$\xi=\sum_{\pi\in P_2(k)}\alpha_\pi\bar{T}_\pi$$

By using the above M\"obius formula, we obtain:
\begin{eqnarray*}
\xi
&=&\sum_{\pi\in P_2(k)}\alpha_\pi\sum_{\sigma\leq\tau\leq\pi}\varepsilon(\tau)\mu(\sigma,\tau)T_\sigma\\
&=&\sum_{\sigma\in P_{even}(k)}T_\sigma\sum_{\pi\in P_2(k)}\left(\sum_{\sigma\leq\tau\leq\pi}\varepsilon(\tau)\mu(\sigma,\tau)\right)\alpha_\pi
\end{eqnarray*}

Our assumption that $\xi$ belongs to the span in the middle reads:
$$\sum_{\pi\in P_2(k)}\left(\sum_{\sigma\leq\tau\leq\pi}\varepsilon(\tau)\mu(\sigma,\tau)\right)\alpha_\pi=0\quad,\quad\forall\sigma\in P_{even}(k)-NC_{even}(k)$$

In the case of pairings, $\sigma\in P_2(k)-NC_2(k)$, this formula simplifies, because the condition $\sigma\leq\tau\leq\pi$ can only be satisfied when $\sigma=\tau=\pi$. Thus, we obtain:
$$\alpha_\sigma=0\quad,\quad\forall \sigma\in P_2(k)-NC_2(k)$$

But this shows that $\xi$ belongs to the span on the right, and we are done.
\end{proof}

Summarizing, we have proved that we have $<\bar{O}_N,H_N^+>=O_N^+$, and this is probably quite interesting, in view of the various general questions regarding the liberations.

As already mentioned, we still have one concrete problem to be solved, namely that of computing the level of $H_N\subset\bar{O}_N$. We will be back to this in section 8 below.

As a conclusion, the liberation questions look quite difficult. We believe that a good input might come from the quantum symmetry groups of the finite graphs, and as a general problem here, we have: when is $G^+(X)$ a maximal liberation of $G(X)$?

This does not look obvious at all, and is open even for the empty graph.

\section{Higher orbitals}

In view of the above considerations, we would like to compute the level of $H_N\subset\bar{O}_N$, and of some related inclusions. The notion of level, as constructed in Definition 5.2 above, regards the fixed point spaces $Fix(u^{\otimes k})$, or rather the dimension of these spaces. This is the case in general, but in the quantum permutation group case, that we are interested in here, all this is related as well to the notions of orbits and orbitals.

In short, we would like to study the orbits and orbitals of the various quantum permutation groups that we have. We will need some general theory. First, we have:

\begin{proposition}
Given a subgroup $G\subset S_N$, consider its magic unitary $u=(u_{ij})$, given by $u_{ij}=\chi\{\sigma\in G|\sigma(j)=i\}$. The following conditions are then equivalent:
\begin{enumerate}
\item $\sigma(i_1)=j_1,\ldots,\sigma(i_k)=j_k$, for some $\sigma\in G$.

\item $u_{i_1j_1}\ldots u_{i_kj_k}\neq0$. 
\end{enumerate}
These conditions produce an equivalence relation $(i_1,\ldots,i_k)\sim(j_1,\ldots,j_k)$, and the corresponding equivalence classes are the $k$-orbitals of $G$. 
\end{proposition}

\begin{proof}
The fact that we have indeed an equivalence as in the statement, which produces an equivalence relation, is indeed clear from definitions. 
\end{proof}

In the quantum case, the situation is more complicated. We follow the approach to the orbits and orbitals developed in \cite{bi2}, \cite{lmr}, and in \cite{mrv} as well. We first have:

\begin{proposition}
Let $G\subset S_N^+$ be a closed subgroup, with magic unitary $u=(u_{ij})$, and let $k\in\mathbb N$. The relation $(i_1,\ldots,i_k)\sim(j_1,\ldots,j_k)$ when $u_{i_1j_1}\ldots u_{i_kj_k}\neq0$ is:
\begin{enumerate}
\item Reflexive.

\item Symmetric.

\item Transitive at $k=1,2$.
\end{enumerate} 
\end{proposition}

\begin{proof}
This is basically known from \cite{bi2}, \cite{lmr}, \cite{mrv}, the proof being as follows:

(1) This simply follows by using the counit:
\begin{eqnarray*}
\varepsilon(u_{i_ri_r})=1,\forall r
&\implies&\varepsilon(u_{i_1i_1}\ldots u_{i_ki_k})=1\\
&\implies&u_{i_1i_1}\ldots u_{i_ki_k}\neq0\\
&\implies&(i_1,\ldots,i_k)\sim(i_1,\ldots,i_k)
\end{eqnarray*}

(2) This follows by applying the antipode, and then the involution:
\begin{eqnarray*}
(i_1,\ldots,i_k)\sim(j_1,\ldots,j_k)
&\implies&u_{i_1j_1}\ldots u_{i_kj_k}\neq0\\
&\implies&u_{j_ki_k}\ldots u_{j_1i_1}\neq0\\
&\implies&u_{j_1i_1}\ldots u_{j_ki_k}\neq0\\
&\implies&(j_1,\ldots,j_k)\sim(i_1,\ldots,i_k)
\end{eqnarray*}

(3) This is something more tricky. We need to prove that we have:
$$u_{i_1j_1}\ldots u_{i_kj_k}\neq0\ ,\ u_{j_1l_1}\ldots u_{j_kl_k}\neq0\implies u_{i_1l_1}\ldots u_{i_kl_k}\neq0$$

In order to do so, we use the following formula:
$$\Delta(u_{i_1l_1}\ldots u_{i_kl_k})=\sum_{s_1\ldots s_k}u_{i_1s_1}\ldots u_{i_ks_k}\otimes u_{s_1l_1}\ldots u_{s_kl_k}$$

At $k=1$ the result is clear, because on the right we have a sum of projections, which is therefore strictly positive when one of these projections is nonzero.

At $k=2$ now, the result follows from the following trick, from \cite{lmr}:
\begin{eqnarray*}
&&(u_{i_1j_1}\otimes u_{j_1l_1})\Delta(u_{i_1l_1}u_{i_2l_2})(u_{i_2j_2}\otimes u_{j_2l_2})\\
&=&\sum_{s_1s_2}u_{i_1j_1}u_{i_1s_1}u_{i_2s_2}u_{i_2j_2}\otimes u_{j_1l_1}u_{s_1l_1}u_{s_2l_2}u_{j_2l_2}\\
&=&u_{i_1j_1}u_{i_2j_2}\otimes u_{j_1l_1}u_{j_2l_2}
\end{eqnarray*}

Indeed, we obtain from this that we have $u_{i_1l_1}u_{i_2l_2}\neq0$, as desired.
\end{proof}

In view of the results that we have so far, we can formulate:

\begin{definition}
Given a closed subgroup $G\subset S_N^+$, consider the relation $\sim_k$ defined by $(i_1,\ldots,i_k)\sim(j_1,\ldots,j_k)$ when $u_{i_1j_1}\ldots u_{i_kj_k}\neq0$.
\begin{enumerate}
\item The equivalence classes with respect to $\sim_1$ are called orbits of $G$.

\item The equivalence classes with respect to $\sim_2$ are called orbitals of $G$.
\end{enumerate}
In the case where $\sim_k$ with $k\geq3$ happens to be transitive, and so is an equivalence relation, we call its equivalence classes the algebraic $k$-orbitals of $G$.
\end{definition}

In order to have some non-trivial examples and counterexamples, let us study the group duals. We recall that we have an embedding $\widehat{\mathbb Z}_N\subset S_N^+$, constructed as follows:
$$\widehat{\mathbb Z}_N\simeq\mathbb Z_N\subset S_N\subset S_N^+$$

To be more precise, if we let $w=e^{2\pi i/N}$ and we denote by $g_1,\ldots,g_N$ the elements of $\mathbb Z_N$, the formula of the corresponding magic unitary over $C^*(\mathbb Z_N)$ is as follows:
$$u_{ij}=\frac{1}{N}\sum_kw^{(i-j)k}g_k$$

Now given integers $N_1,\ldots,N_l$, we can make a dual free product of the embeddings $\widehat{\mathbb Z}_{N_i}\subset S_{N_l}^+$, and we obtain an embedding as follows, with $N=N_1+\ldots+N_l$:
$$\widehat{\mathbb Z_{N_1}*\ldots*\mathbb Z_{N_l}}\subset S_N^+$$

Moreover, given any quotient $\mathbb Z_{N_1}*\ldots*\mathbb Z_{N_l}\to\Gamma$, we obtain in this way an embedding $\widehat{\Gamma}\subset S_N^+$. By a result of Bichon in \cite{bi2}, any group dual $\widehat{\Gamma}\subset S_N^+$ appears in this way.

We will assume in what follows, in order to simplify a number of technical aspects, that our quotients $\Gamma$ appear as intermediate subgroups, as follows:
$$\mathbb Z_{N_1}*\ldots*\mathbb Z_{N_l}\to\Gamma\to\mathbb Z_{N_1}\times\ldots\times\mathbb Z_{N_l}$$

For a number of comments on this assumption, in the context of various matrix modelling questions for the quantum permutation groups, we refer to  \cite{bch}, \cite{bfr}.

Now back to our orbital questions, we first have:

\begin{proposition}
Given an intermediate group $\mathbb Z_{N_1}*\ldots*\mathbb Z_{N_l}\to\Gamma\to\mathbb Z_{N_1}\times\ldots\times\mathbb Z_{N_l}$, consider the associated embedding $\widehat{\Gamma}\subset S_N^+$, with $N=N_1+\ldots+N_l$.
\begin{enumerate}
\item The orbits of $\widehat{\Gamma}$ are the sets producing the partition $\{1,\ldots,N\}=A_1\sqcup\ldots\sqcup A_l$ associated to the decomposition $N=N_1+\ldots+N_l$.

\item The orbitals of $\widehat{\Gamma}$ consist of $N_r$ copies of the set $A_r$, for any $r\in\{1,\ldots,l\}$, along with all the sets $A_r\times A_s$, with $r\neq s$.
\end{enumerate}
\end{proposition}

\begin{proof}
In order to prove this result, let us first discuss the case $l=1$. Here the $k$-orbitals in question are simply those for the usual action $\mathbb Z _N\subset S_N$, and there are $N^{k-1}$ such $k$-orbitals, each of them having size $N$. In general now, the proof is as follows:

(1) This is elementary to prove, starting from the above explicit description of the associated magic unitary, and is well-known since \cite{bi2}.

(2) In order to have $u_{i_1j_1}u_{i_2j_2}\neq0$ we must have $u_{i_1j_1}\neq0,u_{i_2j_2}\neq0$, and so $i_1,j_1\in A_r$,  $i_2,j_2\in A_s$, for certain $r,s\in\{1,\ldots,l\}$. We have two cases, as follows:

\underline{$r=s$}. In this case we have $i_1,j_1,i_2,j_2\in A_r$, and so we are reduced to the study of the orbitals for $\mathbb Z_{N_r}\subset S_{N_r}$, where the answer is trivial, as explained above. Thus, we obtain as orbitals $N_r$ copies of the set $A_r$, for any $r\in\{1,\ldots,l\}$, as in the statement.

\underline{$r\neq s$}. In this case, due to the block diagonal structure of the magic matrix $u=(u_{ij})$, the conditions $u_{i_1j_1}\neq0,u_{i_2j_2}\neq0$ automatically imply $u_{i_1j_1}u_{i_2j_2}\neq0$. Thus, we obtain as extra orbitals the sets $A_r\times A_s$ with $r\neq s$, as in the statement.
\end{proof}

Regarding now the higher orbitals, observe that in order to have $u_{i_1j_1}u_{i_2j_2}u_{i_3j_3}\neq0$ we must have $i_1,j_1\in A_r$, $i_2,j_2\in A_s$, $i_3,j_3\in A_t$ for certain $r,s,t\in\{1,\ldots,l\}$. Thus, the problem naturally splits over the partitions $\ker(rst)\in P(3)$, in the sense that indices coming from triples $(rst)$ having different kernels cannot be connected by $\sim_3$.

With this observation in hand, we have the following result:

\begin{proposition}
For a group dual $\widehat{\Gamma}\subset S_N^+$ as above, $\sim_3$ is an equivalence relation on the subsets of indices corresponding to the following partitions:
\begin{enumerate}
\item $\sqcap\hskip-0.7mm\sqcap$. The classes here consist of $N_r^2$ copies of $A_r$, for any $r$.

\item $\sqcap\,|$. The classes here consist of $N_r$ copies of $A_r\times A_s$, for any $r\neq s$.

\item $|\,\sqcap$. The classes here consist of $N_t$ copies of $A_r\times A_t$, for any $r\neq t$.

\item $|\,|\,|$. The classes here consist of the sets $A_r\times A_s\times A_t$, with $r,s,t$ distinct.
\end{enumerate}
As for the remaining partition, $\sqcap\hskip-3.2mm{\ }_|$\ , here the possible classes depend on $\Gamma$. 
\end{proposition}

\begin{proof}
Let us first discuss the group dual $G=\widehat{\mathbb Z_{N_1}\times\ldots\times\mathbb Z_{N_l}}$. This is a classical group, and so its $\sim_k$ relation is indeed transitive, as a consequence of Proposition 6.1. Regarding now its 3-orbitals, in order to have $u_{i_1j_1}u_{i_2j_2}u_{i_3j_3}\neq0$ we must have $i_1,j_1\in A_r$, $i_2,j_2\in A_s$, $i_3,j_3\in A_t$ for certain $r,s,t\in\{1,\ldots,l\}$, and the situation is as follows:

(I) In the case $r=s=t$ we obtain the 3-orbitals for the action $\mathbb Z_{N_r}\subset S_{N_r}$, which consist of $N_r^2$ copies of $A_r$. Thus, we obtain here $N_r^2$ copies of $A_r$, for any $r$.

(II) In the case $r=s\neq t$ we obtain a product of an orbital at $r$, and an orbit at $t$. The cases $r=t\neq s$ and $s=t\neq r$ are similar.

(III) Finally, in the case where the indices $r,s,t$ are pairwise distinct, we have only 1 orbital, namely the whole set $A_r\times A_s\times A_t$.

In the general case now, as in the statement, the computation in case (I) is identical, and gives (1), the computation in the first two cases of (II) is also identical, and gives (2) and (3), and the computation for (III) gives (4).

Finally, regarding the last assertion, this follows by comparing the products and free products of cyclic groups, and this will be explained in detail below.
\end{proof}

Regarding the partition which is left, namely $\sqcap\hskip-3.2mm{\ }_|$\ , we have results here only in the extreme cases, namely the classical and the free case. In fact, for these two groups the computations can be performed for any $k$, the conclusion being as follows:

\begin{proposition}
Consider the following quantum permutation groups: 
$$\widehat{\mathbb Z_{N_1}\times\ldots\times\mathbb Z_{N_l}}\subset \widehat{\mathbb Z_{N_1}*\ldots*\mathbb Z_{N_l}}$$
\begin{enumerate}
\item In both cases, $\sim_k$ is transitive, for any $k$.

\item These two quantum groups are distinguished by their $3$-orbitals. 
\end{enumerate}
\end{proposition}

\begin{proof}
Let us go back to the proof of Proposition 6.5. For $G=\widehat{\mathbb Z_{N_1}*\ldots*\mathbb Z_{N_l}}$, what changes at $k=3$ is what happens in the case (2), and more specifically in the situation $s=t\neq r$. Indeed, since the underlying algebra is no longer commutative, and is in fact a free product, when assuming $i_1,j_1,i_3,j_3\in A_r$ and $i_2,j_2\in A_s$ with $r\neq s$ we have: 
$$u_{i_1j_1}u_{i_2j_2}u_{i_3j_3}\neq0\iff u_{i_1j_1}\neq0,u_{i_2j_2}\neq 0,u_{i_3j_3}\neq0$$

Thus we have an equivalence relation, and the number of orbitals decreases. 

Summing up, we are done with the case $k=3$. Regarding the higher orbitals, their description for $G=\widehat{\mathbb Z_{N_1}\times\ldots\times\mathbb Z_{N_l}}$ is similar to the one at $k=1,2,3$, basically coming by taking products of orbitals for the cyclic actions $\mathbb Z_{N_r}\subset S_{N_r}$. Thus, we obtain in the end, as full collection of $k$-orbitals, a certain disjoint union of products of the sets $A_r$. 

In the free product case, $G=\widehat{\mathbb Z_{N_1}*\ldots*\mathbb Z_{N_l}}$, our claim is that the situation is quite similar. Indeed, given a non-vanishing product $w=u_{i_1j_1}\ldots u_{i_kj_k}$, we must have $u_{i_1j_1}\neq0,\ldots,u_{i_kj_k}\neq 0$. Thus we must have $i_1,j_1\in A_{r_1},\ldots,i_k,j_k\in A_{r_k}$ for certain numbers $r_1,\ldots,r_k\in\{1,\ldots,l\}$. Now if we group the consecutive terms of $w$ at the places where $r_a=r_{a+1}$, we obtain in this way a certain decomposition of type $w=w_1\ldots w_s$, with the $i,j$ indices of the $u_{ij}$ components of consecutive $w_a$ terms belonging to different $A_r$ sets. Now since we are in a free product situation, we have an equivalence as follows:
$$w\neq0\iff w_1\neq0,\ldots,w_s\neq0$$

Thus, in a way which is similar, but not identical, to the one from the classical case, we end up with an equivalence relation, and the corresponding full collection of $k$-orbitals appears as a certain disjoint union of products of the sets $A_r$.
\end{proof}

Generally speaking, we believe that $\sim_3$ is not transitive, in the general group dual case, but we have no counterexample. Some interesting candidates here come from the various examples worked out in the context of matrix modelling questions in \cite{bch}, \cite{bfr}.

Regarding now the quantum permutation group $S_N^+$ itself, we have here:

\begin{theorem}
For the quantum permutation group $S_N^+$, with $N\geq4$, we have
$$(i_1,\ldots,i_k)\sim(j_1,\ldots,j_k)\iff
\begin{cases}
i_1=i_2\iff j_1=j_2&\\
i_2=i_3\iff j_2=j_3&\\
\ldots&\\
i_{k-1}=i_k\iff j_{k-1}=j_k
\end{cases}$$
and so $\sim$ is an equivalence relation, at any $k\in\mathbb N$. The number of orbits is $2^{k-1}$.
\end{theorem}

\begin{proof}
The implication $\implies$ is clear, because if one of the conditions on the right does not hold, we have $u_{i_1j_1}\ldots u_{i_kj_k}=0$, due to a cancellation between consecutive terms.

Conversely now, we have to show that a vanishing formula of type $u_{i_1j_1}\ldots u_{i_kj_k}=0$ can only come from ``trivial reasons'', as in the statement. But this follows by using group duals, and more specifically by using an embedding as follows:
$$\widehat{\mathbb Z_2*\mathbb Z_2}\subset S_4^+\subset S_N^+$$

Finally, the last assertion is clear, because when counting the orbits for $\sim$, at the level of the pairs $(i_1i_2)$ we have one binary choice to be made, namely $i_1=i_2$ vs. $i_2\neq i_2$, then for the pairs $(i_2i_3)$ we have another binary choice, and so on up to a final binary choice, for $(i_{k-1}i_k)$. Thus, we have $k-1$ binary choices, and so $2^{k-1}$ orbits.
\end{proof}

As an interesting consequence, the algebraic 3-orbitals differ for $S_N$ and $S_N^+$:

\begin{proposition}
The algebraic $3$-orbitals for $S_N$ and $S_N^+$ are as follows:
\begin{enumerate}
\item For $S_N$ we have $5$ such orbitals, corresponding to $\sqcap\hskip-0.7mm\sqcap$, $\sqcap\,|$, $|\,\sqcap$, $\sqcap\hskip-3.2mm{\ }_|$\ , $|\,|\,|$.

\item For $S_N^+$ we have $4$ such orbitals, corresponding to $\sqcap\hskip-0.7mm\sqcap$, $\sqcap\,|$, $|\,\sqcap$, ${\ }_|\!{\ }_|\!{\ }_|\hskip-5.9mm{\ }^{.....}$\ .
\end{enumerate}
\end{proposition}

\begin{proof}
For the symmetric group $S_N$, it follows from definitions that the $k$-orbitals are indexed by the partitions $\pi\in P(k)$, as follows:
$$C_\pi=\left\{(i_1,\ldots,i_k)\Big|\ker i=\pi\right\}$$

Regarding now $S_N^+$, the $k$-orbitals are those computed above, and at $k=3$ they can be naturally indexed by the above diagrams, with the last one standing for the fact that the corresponding 3-orbital merges the $\sqcap\hskip-3.2mm{\ }_|$\ \ and $|\,|\,|$ 3-orbitals from the classical case.
\end{proof}

\section{Analytic orbitals}

Generally speaking, we believe that under suitable ``uniformity'' assumptions, covering the classical case, plus the examples in Proposition 6.6 and Theorem 6.7, and probably many other examples, which still remain to be found, $\sim_3$ should be an equivalence relation, and that the corresponding theory of algebraic $k$-orbitals is worth developing.

However, the fact that the 3-orbitals for $S_N^+$ do not coincide with those for $S_N$ is quite problematic for us, due to a number of reasons explained below. And, our feeling is that the same kind of phenomenon might appear for $H_N^+,\bar{O}_N$ as well. So, it is perhaps better at this point to stop with the algebraic theory, and use instead an analytic approach. 

Let us begin with the following standard result:

\begin{proposition}
For a subgroup $G\subset S_N$, which fundamental corepresentation denoted $u=(u_{ij})$, the following numbers are equal:
\begin{enumerate}
\item The number of $k$-orbitals.

\item The dimension of space $Fix(u^{\otimes k})$.

\item The number $\int_G\chi^k$, where $\chi=\sum_iu_{ii}$.
\end{enumerate}
\end{proposition}

\begin{proof}
This is well-known, the proof being as follows:

$(1)=(2)$ Given $\sigma\in G$ and vector $\xi=\sum_{i_1\ldots i_k}\alpha_{i_1\ldots i_k}e_{i_1}\otimes\ldots\otimes e_{i_k}$, we have:
\begin{eqnarray*}
\sigma^{\otimes k}\xi&=&\sum_{i_1\ldots i_k}\alpha_{i_1\ldots i_k}e_{\sigma(i_1)}\otimes\ldots\otimes e_{\sigma(i_k)}\\
\xi&=&\sum_{i_1\ldots i_k}\alpha_{\sigma(i_1)\ldots\sigma(i_k)}e_{\sigma(i_1)}\otimes\ldots\otimes e_{\sigma(i_k)}
\end{eqnarray*}

Thus $\sigma^{\otimes k}\xi=\xi$ holds for any $\sigma\in G$ precisely when $\alpha$ is constant on the $k$-orbitals of $G$, and this gives the equality between the numbers in (1) and (2).

$(2)=(3)$ This follows from the Peter-Weyl theory, because $\chi=\sum_iu_{ii}$ is the character of the fundamental corepresentation $u$.
\end{proof}

In the quantum case now, $G\subset S_N^+$, by the general Peter-Weyl type results established by Woronowicz in \cite{wo1}, we still have the following formula:
$$\dim Fix(u^{\otimes k})=\int_G\chi^k$$

The problem is that of understanding the $k$-orbital interpretation of this number. We first have the following result, basically coming from \cite{bi2}, \cite{lmr}:

\begin{proposition}
Given a closed subgroup $G\subset S_N^+$, and a number $k\in\mathbb N$, consider the following linear space:
$$F_k=\left\{\xi\in(\mathbb C^N)^{\otimes k}\Big|\xi_{i_1\ldots i_k}=\xi_{j_1\ldots j_k},\forall(i_1,\ldots,i_k)\sim(j_1,\ldots,j_k)\right\}$$
\begin{enumerate}
\item We have $F_k\subset Fix(u^{\otimes k})$.

\item At $k=1,2$ we have $F_k=Fix(u^{\otimes k})$.

\item In the classical case, we have $F_k=Fix(u^{\otimes k})$.

\item For $G=S_N^+$ with $N\geq4$ we have $F_3\neq Fix(u^{\otimes 3})$.
\end{enumerate}
\end{proposition}

\begin{proof}
The tensor power $u^{\otimes k}$ being the corepresentation $(u_{i_1,\ldots i_k,j_1\ldots j_k})_{i_1\ldots i_k,j_1\ldots j_k}$, the corresponding fixed point space $Fix(u^{\otimes k})$ consists of the vectors $\xi$ satisfying:
$$\sum_{j_1\ldots j_k}u_{i_1j_1}\ldots u_{i_kj_k}\xi_{j_1\ldots j_k}=\xi_{i_1\ldots i_k}\quad,\quad\forall i_1,\ldots,i_k$$

With this formula in hand, the proof goes as follows:

(1) Assuming $\xi\in F_k$, the above fixed point formula holds indeed, because:
$$\sum_{j_1\ldots j_k}u_{i_1j_1}\ldots u_{i_kj_k}\xi_{j_1\ldots j_k}=\sum_{j_1\ldots j_k}u_{i_1j_1}\ldots u_{i_kj_k}\xi_{i_1\ldots i_k}=\xi_{i_1\ldots i_k}$$

(2) This is something more tricky, coming from the following formulae:
$$u_{ik}\left(\sum_ju_{ij}\xi_j-\xi_i\right)=u_{ik}(\xi_k-\xi_i)$$
$$u_{i_1k_1}\left(\sum_{j_1j_2}u_{i_1j_1}u_{i_2j_2}\xi_{j_1j_2}-\xi_{i_1i_2}\right)u_{i_2k_2}=u_{i_1k_1}u_{i_2k_2}(\xi_{k_1k_2}-\xi_{i_1i_2})$$

(3) This follows indeed from Proposition 7.1 above.

(4) This follows from Proposition 6.8 above, and from the representation theory of $S_N^+$ with $N\geq4$, the dimensions of the two spaces involved being $4<5$.
\end{proof}

The above considerations suggest formulating the following definition:

\begin{definition}
Given a closed subgroup $G\subset U_N^+$, the integer
$$\dim Fix(u^{\otimes k})=\int_G\chi^k$$
is called number of analytic $k$-orbitals.
\end{definition}

To be more precise here, in the classical case the situation is of course well understood, and this is the number of $k$-orbitals. The same goes for the general case, with $k=1,2$, where this is the number of $k$-orbitals, as constructed in section 6 above.

At $k=3$ and higher, however, Proposition 7.2 (4) shows that, even in the case where the algebraic $3$-orbitals are well-defined, their number is not necessarily the above one. However, we believe that the above definition is the ``correct'' one.

As a further illustration, let us discuss as well what happens for the group duals. With notations from Proposition 6.4 above, if we denote by $g_{N_r}^1,\ldots,g_{N_r}^r$ the elements of each $\mathbb Z_{N_r}$, or rather the images of these elements inside $\Gamma$, with the order of these elements being irrelevant, the following set satisfies $1\in S=S^{-1}$, and is generating for $\Gamma$:
$$S=\left\{g_{N_r}^{i_r}\Big|r=1,\ldots,l,i_r=1,\ldots,N_r\right\}$$

Thus, we can consider the Cayley graph of $\Gamma$ with respect to this set, and then perform random walks on this graph. With this convention, we have the following result:

\begin{proposition}
For the usual products or free products of cyclic groups, the following numbers coincide:
\begin{enumerate}
\item The number of algebraic $k$-orbitals.

\item The number of analytic $k$-orbitals.

\item The number of $k$-loops based at $1$, on the Cayley graph of $\Gamma$.
\end{enumerate}
\end{proposition}

\begin{proof}
It is well-known, as a consequence of $u\sim diag(S)$, that the numbers in (2) and (3) coincide. Thus, in order to prove the result, we have to compare (1) and (3).

As a first observation, at $k=1,2$ this follows from Proposition 6.4, and this, without product assumptions on $\Gamma$. Indeed, at $k=1$ each set $A_r$ corresponds to the loop $1-1_{N_r}$, and at $k=2$ the $N_r$ copies of $A_r$ correspond to the $N_r$ loops of type $1-g_{N_r}-g_{N_r}g^{-1}_{N_r}$, and the sets $A_r\times A_s$ correspond to the loops $1-1_{N_r}-1_{N_r}1_{N_s}$.

At $k\geq3$ the proof is similar for the classical products and the free products, by using the description of the $k$-orbitals from Proposition 6.6 above. To be more precise, for the classical products this is routine, and follows as well from the fact that we have $(1)=(2)$. As for the free product case, the point here is that, by using the word decomposition $w=w_1\ldots w_s$ from the proof of Proposition 6.6, each of the words $w_a$ must correspond to a certain loop on the Cayley graph, and this gives the result.
\end{proof}

Now back to the definition of the analytic $k$-orbitals, this has of course the advantage of being defined for any $k$. In the particular case $k=3$, we have as well the following result, from \cite{ba3}, which brings some more support for our definition:

\begin{proposition}
For a closed subgroup $G\subset S_N^+$, and an integer $k\leq3$, the following conditions are equivalent:
\begin{enumerate}
\item $G$ is $k$-transitive, in the sense that $Fix(u^{\otimes k})$ has dimension $1,2,5$.

\item The $k$-th moment of the main character is $\int_G\chi^k=1,2,5$.

\item $\int_Gu_{i_1j_1}\ldots u_{i_kj_k}=\frac{(N-k)!}{N!}$ for distinct indices $i_r$ and distinct indices $j_r$.

\item $\int_Gu_{i_1j_1}\ldots u_{i_kj_k}$ equals $\frac{(N-|\ker i|)!}{N!}$ when $\ker i=\ker j$, and equals $0$, otherwise.
\end{enumerate}
\end{proposition}

\begin{proof}
Most of these implications are known since \cite{ba3}, the idea being as follows:

$(1)\iff(2)$ This follows from the Peter-Weyl type theory from \cite{wo1}, because the $k$-th moment of the character counts the number of fixed points of $u^{\otimes k}$.

$(2)\iff(3)$ This follows from the Schur-Weyl duality results for $S_N,S_N^+$ and from $P(k)=NC(k)$ at $k\leq3$, as explained in \cite{ba3}. 

$(3)\iff(4)$ Once again this follows from $P(k)=NC(k)$ at $k\leq3$, and from a standard integration result for $S_N$, as explained in \cite{ba3}. 
\end{proof}

As a conclusion to all these considerations, we have:

\begin{theorem}
For a closed subgroup $G\subset S_N^+$, and an integer $k\in\mathbb N$, the number $\dim(Fix(u^{\otimes k}))=\int_G\chi^k$ of ``analytic $k$-orbitals'' has the following properties:
\begin{enumerate}
\item In the classical case, this is the number of $k$-orbitals.

\item In general, at $k=1,2$, this is the number of $k$-orbitals.

\item For the free products of cyclic groups, this is the number of algebraic $k$-orbitals.

\item At $k=3$, when this number is minimal, $G$ is $3$-transitive in the above sense.
\end{enumerate}
\end{theorem}

\begin{proof}
This follows indeed from the above considerations.
\end{proof}

There are of course many questions left. A first one regards the case $k=4$, where we do not know what the correct analogue of Proposition 7.5 would be. This is of course quite important, because it would bring more support for our definition at $k=4$.

A second question regards the interpretation of $\int_G\chi^k$, as ``counting'' certain objects, that we can call afterwards ``$k$-orbitals''. In the case $G=S_N^+$ we have $\int_G\chi^k=\#NC(k)$, so these $k$-orbitals that we are looking for can only be the elements of $NC(k)$, in some index-theoretic formulation. However, all this heavily relies on the easiness property of $S_N^+$, and for other quantum groups it is not clear what the ``candidates'' should be.

We believe, however, that this latter question can be subject to some further investigation. From an analytic perspective, the relevant formula is:
$$\int_G\chi^k=\sum_{i_1\ldots i_k}\int_Gu_{i_1i_1}\ldots u_{i_ki_k}$$

The problem is to understand how the integrals on the right can be naturally grouped into sums which are integers. This question can be probably investigated by using the Weingarten formula \cite{bco}, but we have no further results here.

\section{Reflection groups}

Let us go back now to the quantum group $\bar{O}_N$, and the other quizzy quantum groups. As explained in sections 4 and 5 above, we have two types of actions to be investigated, namely $H_N\subset\bar{O}_N\subset S_{2^N}^+$ and $H_N\subset H_N^+\subset S_{2N}^+$. We will compute here the small order higher orbitals of these actions, in the analytic sense explained in section 7 above.

Let us first study the actions of $H_N\subset\bar{O}_N\subset S_{2^N}^+$. We have here:

\begin{proposition}
The orbitals for $H_N\subset\bar{O}_N\subset S_{2^N}^+$ are as follows:
\begin{enumerate}
\item At $k=1$ we have $1$ orbit, in both cases.

\item At $k=2$ we have $N+1$ orbitals, in both cases.
\end{enumerate}
\end{proposition}

\begin{proof}
Indeed, the action of $H_N$ on the hypercube is transitive, and has $N+1$ orbitals, corresponding to the diagonals of the cube having lengths $\sqrt{0},\sqrt{1},\sqrt{2},\ldots,\sqrt{N}$. 

Regading now $\bar{O}_N$, we know from section 4 that the magic character decomposes as $\chi=\chi_0+\ldots+\chi_N$, and it follows that we have $\int\chi^2\geq N+1$. Now since for $H_N$ the corresponding integral equals $N+1$, for $\bar{O}_N$ we must obtain $N+1$ as well, as stated.
\end{proof}

Regarding now the higher $k$-orbitals for the action $H_N\subset S_{2^N}$, these appear from the $k$-simplices having the vertices on the standard cube, and so having edges of lengths $\sqrt{0},\sqrt{1},\sqrt{2},\ldots,\sqrt{N}$, which each simplex appearing with a certain multiplicity. 

As for the quantum group $\bar{O}_N$, we can use the correspondence with $O_N$, and we are therefore led to questions regarding the antisymmetric representation of $O_N$. Thus, we can in principle compute the number of $k$-orbitals by using the Weingarten formula.

Both computations are non-trivial, and as a conclusion here, we have: 

\begin{conjecture}
The quantum groups $H_N\subset\bar{O}_N$ are distinguished by their $3$-orbitals.
\end{conjecture}

Some good evidence for this statement comes from the fact that at $k=3$ the problem for $O_N$ looks purely combinatorial, while the problem for $H_N$ involves some analysis, coming from triangle inequalities for the edges of the triangles. Thus, the combinatorics is not the same, and so the results of the computations should be different.

Let us study now the actions $H_N\subset H_N^+\subset S_{2N}^+$, where the computations are considerably simpler. First, we have the following elementary result:

\begin{proposition}
The orbitals for $H_N\subset S_{2N}$ are as follows:
\begin{enumerate}
\item At $k=1$ we have $1$ orbit.

\item At $k=2$ we have $3$ orbitals.

\item At $k=3$ we have $11$ orbitals.

\item At $k=4$ we have $49$ orbitals.
\end{enumerate}
\end{proposition}

\begin{proof}
We recall that the action $H_N\subset S_{2N}$ comes by permuting $N$ segments. Thus the $k$-orbitals for $H_N$ are obtained by decorating the $2N$ endpoints of these $N$ segments with $k$ dots, and then by counting the multiplicity of each configuration. 

At small values of $k$, the situation is as follows:

(1) Here the action is clearly transitive, and the corresponding 1 orbital appears from the only possible configuration, namely $\bullet$---\ , appearing once. 

(2) Here the possible configurations are $\bullet\bullet$---\ ,\ $\bullet$---$\bullet$\ ,\ $\bullet$---\, $\bullet$---\ , with each appearing exactly once, and so we have 3 orbitals.

(3) Here the possible configurations, with multiplicities, are $\bullet\bullet\bullet$---\ ($\times1$),\ $\bullet\bullet$---$\bullet$\ ($\times3$),\ $\bullet\bullet$---\, $\bullet$---\ ($\times3$), $\bullet$---$\bullet$\ $\bullet$---\ ($\times3$), $\bullet$---\ $\bullet$---\ $\bullet$---\ ($\times1$), and we have 11 orbitals.

(4) Here the possible configurations, with multiplicities, are $\bullet\bullet\bullet\;\bullet$---\ ($\times1$),\ $\bullet\bullet\bullet$---$\bullet$\ ($\times4$),\ $\bullet\bullet$---$\bullet\bullet$\ ($\times3$)\ , $\bullet\bullet\bullet$---\, $\bullet$---\ ($\times4$), $\bullet\bullet$---$\bullet$\ $\bullet$---\ ($\times12$), $\bullet\bullet$---\ $\bullet\bullet$---\ ($\times3$), $\bullet\bullet$---\ $\bullet$---$\bullet$\ ($\times6$), $\bullet$---$\bullet$\ $\bullet$---$\bullet$\ ($\times3$), $\bullet\bullet$---\ $\bullet$---\ $\bullet$---\ ($\times6$), $\bullet$---$\bullet$\ $\bullet$---\ $\bullet$---\ ($\times6$), $\bullet$---\ $\bullet$---\ $\bullet$---\ $\bullet$---\ ($\times1$), and so we have a total of $49$ orbitals.
\end{proof}

In order to deal now with $H_N^+\subset S_{2N}^+$, we recall from \cite{bbc} that we have:

\begin{proposition}
When using the standard embeddings $S_N^+\subset H_N^+\subset O_N^+$, the magic matrix for the embedding $H_N^+\subset S_{2N}^+$ comes from the ``sudoku'' matrix 
$$v=\begin{pmatrix}a&b\\ b&a\end{pmatrix}$$
the connecting formulae between $u,v$ being $u_{ij}=a_{ij}-b_{ij}$ and $a_{ij}=(u_{ij}^2+u_{ij})/2$, $b_{ij}=(u_{ij}^2-u_{ij})/2$. In addition, the matrix $p_{ij}=u_{ij}^2$ is magic, and we have $p\in u^{\otimes 2}$.
\end{proposition}

\begin{proof}
The first assertion, which is similar with what happens in the classical case, where we have $u_{ij}\in\{-1,0,1\}$, is explained in detail in \cite{bbc}. Regarding now the second assertion, which follows in fact from the general results in \cite{bve}, observe that we have:
$$u^{\otimes 2}(e_i\otimes e_i)
=\sum_{kl}u_{ki}u_{li}e_k\otimes e_l
=\sum_ku_{ki}^2e_k\otimes e_k
=\sum_kp_{ki}e_k\otimes e_k$$

Thus the linear space $span(e_i\otimes e_i)$ is left invariant by $u^{\otimes 2}$, and the corresponding subrepresentation of $u^{\otimes 2}$ is the magic corepresentation $p=(p_{ij})$, as claimed.
\end{proof}

By using this description, we obtain the following result:

\begin{proposition}
The quantum groups $H_N\subset H_N^+$ have the same number of $1,2,3$-orbitals, when regarded as subgroups of $S_{2N}^+$, namely $1,3,11$.
\end{proposition}

\begin{proof}
We use the formula in Proposition 8.4, which gives:
$$\chi_v=2\sum_ia_{ii}=\sum_iu_{ii}^2+u_{ii}=\chi_p+\chi_u$$

With this formula in hand, the proof goes as follows:

(1) $k=1$. Here there is nothing to prove, because since $H_N$ is transitive, so must be $H_N^+$. Observe that this follows as well from $\int\chi_p=1$, $\int\chi_u=0$.

(2) $k=2$. Here we can use once again an elementary argument. First, since $H_N$ has 3 orbitals, $H_N^+$ can have 2 or 3 orbitals. Now observe that we have: 
$$v_{11}v_{N+1,N+2}=a_{11}a_{12}=0$$

But this shows that we have $(1,N+1)\not\sim(1,N+2)$, and so the action of $H_N^+$ is not doubly transitive, and so we must have 3 orbitals, as claimed.

Note that this follows as well from the following computation:
$$\int\chi_v^2=\int\chi_p^2+2\int\chi_p\chi_u+\int\chi_u^2=2+0+1=3$$

(3) $k=3$. Here we cannot use direct algebraic arguments, because the algebraic 3-orbitals, even if they exist, are not nesessarily counted by the moments of the main character. Thus, we must integrate characters. Since $p$ is magic, we have:
$$\int\chi_p^2\geq2\quad,\quad\int\chi_p^3\geq5$$

Indeed, this follows from the representation theory of $S_N^+$, because $2,5$ are Bell numbers, counting respectively the partitions in $NC(2),NC(3)$. Now by using as well the fact that we have $p\in u^{\otimes2}$, from Proposition 8.4 above, we obtain that we have:
\begin{eqnarray*}
\int\chi_v^3
&=&\int\chi_p^3+3\int\chi_p^2\chi_u+3\int\chi_p\chi_u^2+\int\chi_u^3\\
&\geq&\int\chi_p^3+3\int\chi_p^2\chi_u+3\int\chi_p^2+\int\chi_u^3\\
&\geq&5+3\times0+3\times2+0\\
&=&11
\end{eqnarray*}

On the other hand, we know from Proposition 8.3 above that for $H_N$, the corresponding integral is 11. Thus, by functoriality, we obtain 11, as claimed.
\end{proof}

Thus, we must use $k=4$ in order to distinguish $H_N\subset H_N^+$. We have:

\begin{theorem}
The quantum groups $H_N\subset H_N^+\subset S_{2N}^+$ are not distinguished by their $1,2,3$-orbitals, but have different numbers of $4$-orbitals, namely $49>43$.
\end{theorem}

\begin{proof}
By using the character formula $\chi_v=\chi_p+\chi_u$ from the proof of Proposition 8.3 above, and the trace property of the integration functional, we obtain:
$$\int\chi_v^4=\int\chi_p^4+4\int\chi_p^3\chi_u+\left(4\int\chi_p^2\chi_u^2+2\int\chi_u\chi_p\chi_u\chi_p\right)+4\int\chi_p\chi_u^3+\int\chi_u^4$$

Our claim is that the difference between $H_N,H_N^+$ comes from the quantity in the middle, which must decrease when $\chi_p,\chi_u$ do not commute. Indeed:

(1) Regarding $H_N$, we know from Proposition 8.3 that the result holds indeed. This can be recovered as well by using the above integral, as follows:
$$\int\chi_v^4=15+4\times 0+(4+2)\times 5+4\times0+4=49$$

Here we have used the fact that we have $1\in p\in u^{\otimes2}$, which gives:
$$\int\chi_p^2\chi_u^2=\#(1\in u^{\otimes 2}\otimes p^{\otimes 2})=\#(1\in p^{\otimes 3})=5$$

(2) Regarding now $H_N^+$, the point is that the quantity $\int\chi_p^2\chi_u^2$ can be computed as above, but the quantity $\int\chi_u\chi_p\chi_u\chi_p$ is no longer equal to it. In order to compute this latter quantity, observe that by using $p\in u^{\otimes2}$, and then $1\in p$, we obtain:
\begin{eqnarray*}
p^{\otimes2}
&\in&u^{\otimes 2}\otimes p\\
&=&u\otimes1\otimes u\otimes p\\
&\in&u\otimes p\otimes u\otimes p
\end{eqnarray*}

We therefore obtain the following estimate:
$$\int(\chi_u\chi_p)^2\geq\int\chi_p^2=2$$

On the other hand, it follows from the fusion rules computed in \cite{bve} that the reverse inequality holds as well. Thus, we have $2\times 3=6$ orbitals missing with respect to the $H_N$ case, and so we have a total of $49-6=43$ orbitals for $H_N^+$, as stated.
\end{proof}


\begin{thebibliography}{99}

\bibitem{ba1}T. Banica, Liberations and twists of real and complex spheres, {\em J. Geom. Phys.} {\bf 96} (2015), 1--25.

\bibitem{ba2}T. Banica, A duality principle for noncommutative cubes and spheres, {\em J. Noncommut. Geom.} {\bf 10} (2016), 1043--1081.

\bibitem{ba3}T. Banica, Higher transitive quantum groups: theory and models, {\em Colloq. Math.}, to appear.

\bibitem{ba4}T. Banica, Quantum groups, from a functional analysis perspective, {\em Adv. Oper. Theory} {\bf 4} (2019), 164--196. 

\bibitem{bbc}T. Banica, J. Bichon and B. Collins, The hyperoctahedral quantum group, {\em J. Ramanujan Math. Soc.} {\bf 22} (2007), 345--384.

\bibitem{bch}T. Banica and A. Chirvasitu, Quasi-flat representations of uniform groups and quantum groups, {\em J. Algebra Appl.}, to appear. 

\bibitem{bco}T. Banica and B. Collins, Integration over quantum permutation groups, {\em J. Funct. Anal.} {\bf 242} (2007), 641--657.

\bibitem{bfr}T. Banica and A. Freslon, Modelling questions for quantum permutations, {\em Infin. Dimens. Anal. Quantum Probab. Relat. Top.} {\bf 21} (2018), 1--26.

\bibitem{bsp}T. Banica and R. Speicher, Liberation of orthogonal Lie groups, {\em Adv. Math.} {\bf 222} (2009), 1461--1501.

\bibitem{bve}T. Banica and R. Vergnioux, Fusion rules for quantum reflection groups, {\em J. Noncommut. Geom.} {\bf 3} (2009), 327--359.

\bibitem{bpa}H. Bercovici and V. Pata, Stable laws and domains of attraction in free probability theory, {\em Ann. of Math.} {\bf 149} (1999), 1023--1060.

\bibitem{bgo}J. Bhowmick and D. Goswami, Quantum isometry groups: examples and computations, {\em Comm. Math. Phys.} {\bf 285} (2009), 421--444.

\bibitem{bi1}J. Bichon, Free wreath product by the quantum permutation group, {\em Alg. Rep. Theory} {\bf 7} (2004), 343--362.

\bibitem{bi2}J. Bichon, Algebraic quantum permutation groups, {\em Asian-Eur. J. Math.} {\bf 1} (2008), 1--13.

\bibitem{bi3}J. Bichon, Half-liberated real spheres and their subspaces, {\em Colloq. Math.} {\bf 144} (2016), 273--287.

\bibitem{bdu}J. Bichon and M. Dubois-Violette, Half-commutative orthogonal Hopf algebras, {\em Pacific J. Math.} {\bf 263} (2013), 13--28. 

\bibitem{bra}R. Brauer, On algebras which are connected with the semisimple continuous groups, {\em Ann. of Math.} {\bf 38} (1937), 857--872.

\bibitem{csn}B. Collins and P. \'Sniady, Integration with respect to the Haar measure on unitary, orthogonal and symplectic groups, {\em Comm. Math. Phys.} {\bf 264} (2006), 773--795.

\bibitem{dri}V.G. Drinfeld, Quantum groups, Proc. ICM Berkeley (1986), 798--820.

\bibitem{fr1}A. Freslon, On the partition approach to Schur-Weyl duality and free quantum groups, {\em Transform. Groups} {\bf 22} (2017), 707--751.

\bibitem{fr2}A. Freslon, On two-coloured noncrossing quantum groups, preprint 2017.

\bibitem{gos}D. Goswami, Existence and examples of quantum isometry groups for a class of compact metric spaces, {\em Adv. Math.} {\bf 280} (2015), 340--359.

\bibitem{gro}D. Gromada, Classification of globally colorized categories of partitions, preprint 2018.

\bibitem{jim}M. Jimbo, A $q$-difference analog of $U(\mathfrak g)$ and the Yang-Baxter equation, {\em Lett. Math. Phys.} {\bf 10} (1985), 63--69.

\bibitem{lmr}M. Lupini, L. Man\v cinska and D.E. Roberson, Nonlocal games and quantum permutation groups, preprint 2017.

\bibitem{mal}S. Malacarne, Woronowicz's Tannaka-Krein duality and free orthogonal quantum groups, {\em Math. Scand.} {\bf 122} (2018), 151--160.

\bibitem{mrv}B. Musto, D.J. Reutter and D. Verdon, A compositional approach to quantum functions, preprint 2017.

\bibitem{ntu}S. Neshveyev and L. Tuset, Compact quantum groups and their representation categories, SMF (2013).

\bibitem{rwe}S. Raum and M. Weber, The full classification of orthogonal easy quantum groups, {\em Comm. Math. Phys.} {\bf 341} (2016), 751--779.

\bibitem{twe}P. Tarrago and M. Weber, Unitary easy quantum groups: the free case and the group case, {\em Int. Math. Res. Not.} {\bf 18} (2017), 5710--5750.

\bibitem{vdn}D.V. Voiculescu, K.J. Dykema and A. Nica, Free random variables, AMS (1992).

\bibitem{wa1}S. Wang, Free products of compact quantum groups, {\em Comm. Math. Phys.} {\bf 167} (1995), 671--692.

\bibitem{wa2}S. Wang, Quantum symmetry groups of finite spaces, {\em Comm. Math. Phys.} {\bf 195} (1998), 195--211.

\bibitem{wei}D. Weingarten, Asymptotic behavior of group integrals in the limit of infinite rank, {\em J. Math. Phys.} {\bf 19} (1978), 999--1001.

\bibitem{wo1}S.L. Woronowicz, Compact matrix pseudogroups, {\em Comm. Math. Phys.} {\bf 111} (1987), 613--665.

\bibitem{wo2}S.L. Woronowicz, Tannaka-Krein duality for compact matrix pseudogroups. Twisted SU(N) groups, {\em Invent. Math.} {\bf 93} (1988), 35--76.

\end{thebibliography}
\end{document}